\documentclass[a4paper,11pt,reqno]{amsart}

\usepackage{parskip}
\usepackage{amsthm,amsmath,amsfonts,amssymb}
\usepackage{cases}
\usepackage{xfrac}

\usepackage{color}
\usepackage{graphicx}
\usepackage{subcaption}
\usepackage{pgf, tikz-cd}
\usetikzlibrary{spy}
\usetikzlibrary{arrows}
\usepackage{mathrsfs}
\usepackage{tikzscale}
\usepackage[percent]{overpic}

\definecolor{ffffff}{rgb}{1.,1.,1.}
\definecolor{cqcqcq}{rgb}{0.75,0.75,0.75}

\usepackage[
colorlinks=true,
urlcolor=teal,
citecolor=red,
linkcolor=teal,
linktocpage,
pdfpagelabels,
bookmarksnumbered,
bookmarksopen]
{hyperref}
\usepackage[
capitalize, 
nameinlink,
noabbrev]
{cleveref}

\newcommand{\e}{\varepsilon}

\renewcommand{\k}{\kappa}

\newcommand{\g}{\gamma}

\renewcommand{\t}{\tau}

\renewcommand{\phi}{\varphi}

\newcommand{\G}{\Gamma}
\newcommand{\Om}{\Omega}

\newcommand{\cF}{\mathcal{F}}

\newcommand{\R}{\mathbb{R}}

\DeclareMathOperator{\dist}{dist}
\DeclareMathOperator{\interior}{int}
\DeclareMathOperator{\reach}{reach}
\DeclareMathOperator{\inr}{inr}
\DeclareMathOperator{\argmin}{arg\,min}
\DeclareMathOperator{\argmax}{arg\,max}

\newcommand{\de}{\partial}

\theoremstyle{plain}
\newtheorem{thm}{Theorem}[section]
\newtheorem*{thm*}{Theorem}
\newtheorem{lem}[thm]{Lemma}
\newtheorem{prop}[thm]{Proposition}
\newtheorem{cor}[thm]{Corollary}

\theoremstyle{definition}
\newtheorem{defin}[thm]{Definition}

\theoremstyle{remark}
\newtheorem{rem}[thm]{Remark}

\numberwithin{equation}{section}

\title[The isoperimetric problem in $2$d domains...]{The isoperimetric problem\\ in $2$d domains without necks}

\author[G.~P.~Leonardi]{Gian Paolo Leonardi}
\address[Gian Paolo Leonardi]{Dipartimento di Matematica, Universit\`a di Trento, via Sommarive 14, IT--38123 Povo-Trento}
\email{gianpaolo.leonardi@unitn.it}
\author[G.~Saracco]{Giorgio Saracco}
\address[Giorgio Saracco]{Dipartimento di Matematica, Universit\`a di Trento, via Sommarive 14, IT--38123 Povo-Trento}
\email{giorgio.saracco@unitn.it}%

\thanks{G.P.L.~and G.S.~are members of INdAM and have been partially supported by the INdAM--GNAMPA Project 2020 ``Problemi isoperimetrici con anisotropie'' (n.~prot.~U-UFMBAZ-2020-000798 15-04-2020).}



\begin{document}

\begin{abstract}
We give a complete characterization of all isoperimetric sets contained in a domain of the Euclidean plane, that is bounded by a Jordan curve and satisfies a no-neck property. Further, we prove that the isoperimetric profile of such domain is convex above the volume of the largest ball contained in it, and that its square is globally convex.
\end{abstract}

 \hspace{-3cm}
 {
 \begin{minipage}[t]{0.6\linewidth}
 \begin{scriptsize}
 \vspace{-3cm}
 This is a pre-print of an article published in \emph{Calc.~Var.~Partial Differential Equations}. The final authenticated version is available online at: \href{https://doi.org/10.1007/s00526-021-02153-9}{https://doi.org/10.1007/s00526-021-02153-9}
 \end{scriptsize}
\end{minipage} 
}

\maketitle

\section{Introduction}

Given a bounded, open set $\Omega\subset \R^{n}$, $n\ge 2$, we consider the isoperimetric problem among Borel subsets of $\Omega$, that is, the minimization of the perimeter $P(E)$ of a Borel set $E\subset \Omega$ subject to a volume constraint $|E|=V$, where by perimeter we mean the distributional one in the sense of Caccioppoli--De Giorgi, and where $|E|$ denotes the Lebesgue measure of $E$ and $V\in [0,|\Omega|]$. Moreover, we are interested in the properties of the total isoperimetric profile
\begin{equation}\label{eq:isop_prof}
\mathcal{J}(V):=\inf\left\{\,P(E)\,:\, |E|=V,\, E\subset \Omega \,\right\}\,,
\end{equation}
as a function defined on $[0,|\Omega|]$. If $R_\Om$ denotes the inradius of $\Om$ (i.e., the radius of the largest ball contained in $\Om$) and if $\omega_{n}$ represents the Lebesgue measure of the unit ball in $\R^{n}$, then the classical isoperimetric inequality in $\R^{n}$ implies that the unique minimizers for volumes $0<V\le \omega_{n}R_{\Omega}^{n}$ are balls, up to null sets. Thus, finding and characterizing minimizers, as well as computing $\mathcal{J}(V)$, is a trivial problem whenever $V\le \omega_n R^n_\Om$, while it becomes a challenging problem for larger $V$. 

By well-known compactness and semicontinuity properties of the perimeter, proving the existence of minimizers is a quite straightforward task. An alternative approach to solve \eqref{eq:isop_prof} is through the minimization of the unconstrained problem
\begin{equation}\label{eq:pmc}
\mathcal{F}_\kappa[E] := P(E)-\k|E|\,,
\end{equation}
where $\k>0$ is a fixed constant and $E\subset \Omega$. The functional $\mathcal{F}_\k$ is usually referred to as the \emph{prescribed mean curvature functional}, since any nontrivial minimizer $E_\kappa$ is such that $\de E_\kappa \cap \Omega$ is analytic up to a closed singular set of Hausdorff dimension at most $n-8$, and has constant mean curvature equal to $(n-1)^{-1}\kappa$. Whenever a set $E_\k$ minimizes $\mathcal{F}_\k$, it is as well a minimizer of~\eqref{eq:isop_prof} for the prescribed volume $V=|E_\k|$. Constructing minimizers of~\eqref{eq:isop_prof} via the unconstrained problem~\eqref{eq:pmc} is a viable strategy only when the constant $\k$ is chosen greater than or equal to the Cheeger constant of $\Om$\begin{equation}\label{eq:cheeger}
h_\Om:= \inf \left\{\, \frac{P(E)}{|E|}\,:\, E\subset \Om\, \right\}\,.
\end{equation}
Indeed, when $\k<h_\Om$ the functional~\eqref{eq:pmc} has the empty set as the unique minimizer, therefore we gain no useful information in this case. In fact, this program has been carried out in~\cite{ACC05} in the $n$-dimensional case for convex, $C^{1,1}$ regular sets $\Om$, and in our recent paper~\cite{LS20} in the $2$d case for a special class of simply-connected domains that includes all (open, bounded) convex sets. Any minimizer of~\eqref{eq:cheeger} is referred to as a \emph{Cheeger set} of $\Om$. In the settings of~\cite{ACC05, LS20}, among all Cheeger sets of $\Om$, the ones with least and greatest volumes (the so-called \textit{minimal} and \textit{maximal} Cheeger sets) are unique, and we shall denote them, respectively, by $E^m_{h_\Om}$ and $E^M_{h_\Om}$ (in the setting of~\cite{ACC05} it is a consequence of~\cite{AC09}, while we refer to~\cite[Theorem~2.3]{LS20} for the other setting). Then, it is shown in \cite{ACC05, LS20} that for all volumes $V$ greater than or equal to the volume of the minimal Cheeger set $|E^m_{h_\Om}|$, one can find a curvature $\k$ and a minimizer $E_\k$ of~\eqref{eq:pmc} such that $|E_\k|=V$, and thus $P(E_\k)=\mathcal{J}(V)$.

Unless $\Om$ is itself a ball, one always has the strict inequality $\omega_n R^n_\Om<|E^m_{h_\Om}|$, and one can easily exhibit sets for which $|E^m_{h_\Om}|-\omega_n R^n_\Om$ is as big as one wishes, see \cref{sec:examples} for some examples in dimension $n=2$. Thus, there is a possibly very wide range of volumes for which we cannot tackle directly the isoperimetric problem through the unconstrained minimization of $\mathcal{F}_\k$ for suitable values of $\k$. In particular, this approach fails because the functional $\mathcal{F}_\k$ is uniquely minimized by the empty set for values $\k< h_\Om$. However, by suitably shrinking the class of competitors we can altogether avoid this problem. Namely, we consider the minimization problem
\begin{equation}\label{eq:pmc_mod}
\inf\{\,\mathcal{F}_\kappa[E]:=P(E)-\kappa|E|\,: E\in \mathcal{C}(\k)\,\}\,,
\end{equation}
where the class of competitors $\mathcal{C}(\k)$ is set as follows:
\begin{equation}\label{eq:classCk}
 \hspace{-6pt}\mathcal{C}(\k) \hspace{-1.5pt}= \hspace{-1.5pt}
\begin{cases}
E\subset \Om\quad \text{Borel}\,, \quad &\text{if $\k> h_\Om$}\,, \\
E\subset \Om\quad \text{Borel s.t.}\ |E|\ge (n-1)^{n}\omega_n \k^{-n}\,, &\text{if $\frac{n-1}{R_\Om}\le \k \le h_\Om$}\,.
\end{cases}
\end{equation}
In principle, one could also consider a mean curvature $\frac{\k}{n-1}$ smaller than ${R_\Om}^{-1}$, as long as $\omega_n(n-1)^{n} \kappa^{-n}\le |\Om|$, which ensures that $\mathcal{C}(\k) \neq \emptyset$. Nevertheless, the choice $R_\Om^{-1}\le \frac{\k}{n-1}$ is key to get full information on the mean curvature of the minimizers, see \cref{rem:k_piu_piccoli}.

We shall thus consider the minimization problem~\eqref{eq:pmc_mod}, when $\Om$ is a $2$d set whose boundary is a Jordan curve with zero $2$-dimensional Lebesgue measure, and such that $\Om$ has no necks of radius $r$ for all $r\le h_\Om^{-1}$ (see \cref{def:no_necks}), which corresponds to the setting of our previous paper~\cite{LS20}.  In particular, for $\k \ge h_\Om$ we recover the results\footnote{In this former paper we considered $C(h_\Om)=\{E\subset \Om\}$, but we explicitly ruled out from our results the empty set, considering as minimizers of~\eqref{eq:pmc} for $\k=h_\Om$ only those of~\eqref{eq:cheeger}. By changing the class of competitors and considering~\eqref{eq:pmc_mod}, we now avoid treating $\k=h_\Om$ as a special case.} of~\cite{LS20}. In the case $R^{-1}_\Om \le \k< h_\Om$, the addition of a lower bound on the volume as an extra constraint prevents the empty set, and in general any set of small volume, from being a minimizer. One of the core results of this paper is \cref{thm:main_shape_min}, which shows that \emph{all} minimizers of~\eqref{eq:pmc_mod}, for any fixed $\k\ge R_{\Om}^{-1}$, are geometrically characterized as ``suitable unions of balls of radius $\k^{-1}$ contained in $\Om$''. This characterization extends~\cite[Theorem~2.3]{LS20}, where it was proved only for minimizers of $\mathcal{F}_\k$ with the least and the greatest volumes, and under the restrictive assumption $\k \ge h_\Om$.

This geometric characterization allows us to find, for any given volume $V\ge \pi R^2_\Om$, a curvature $\k \ge R_\Om^{-1}$ and a minimizer $E_\k$ of~\eqref{eq:pmc_mod} such that $|E_\k|=V$. Hence, we can characterize all minimizers of~\eqref{eq:isop_prof} and provide an extension of~\cite[Theorem~3.32]{SZ97}, which was proved only for convex sets in dimension $n=2$. For the sake of completeness, we recall that this latter result was not completely new at the time: the dual problem (maximize volume under a perimeter constraint among subsets of a convex, $2$d set) had been first discussed in~\cite[Variant~III]{Bes52} under the assumption of convexity of minimizers, which was later shown to be redundant in~\cite{SS78} (for $P \le P(\Om)$). Additionally, the characterization~\cite[Theorem~3.32]{SZ97} was known for triangles since the papers of Steiner~\cite{Ste42a, Ste42b}, see also~\cite{Dem75}, and for circumscribed polygons~\cite{Lin77}.

Our approach of building isoperimetric sets as minimizers of $\mathcal{F}_\k$ for a suitable $\k$ also allows us to prove some convexity properties of both the isoperimetric profile $\mathcal{J}$, which we show to be convex for $V\ge \pi R^2_\Om$, and its square $\mathcal{J}^2$, which we show to be globally convex. Besides the trivial fact that the total isoperimetric profile is concave up to $V\le \pi R^2_\Om$, we are not aware of any result in the literature concerning its convexity properties until the last year, when few results were first proved. Indeed, in~\cite[Section~3]{FLSS19} it was proved that a suitable relaxation of~\eqref{eq:isop_prof} is convex, while in our previous paper~\cite[Section~6]{LS20}, we proved that, for the same class of domains now under consideration, there exists a threshold volume $\overline{V}$ (which is the volume of the minimal Cheeger set of $\Om$, $|E^m_{h_\Om}|$), above which the isoperimetric profile is convex. Essentially, here we prove that one can lower this threshold all the way down to $\pi R_\Om^2$. This finally means that $\mathcal{J}$ is convex for all volumes $V\ge \pi R_\Om^2$. In some sense, the presence of the boundary of $\Om$ as an obstacle forces the isoperimetric profile to switch from concave to convex, in the range of volumes where the interaction between isoperimetric sets and the obstacle $\de \Om$ becomes effective.

Finally, it is worth noting that our proof of convexity does not rely on knowing the shape of minimizers, but rather on knowing that for all volumes above a certain threshold $\overline{V}$ one can find isoperimetric sets as minimizers of $\mathcal{F}_\k$ for a suitable $\k$. Thanks to the results in~\cite{ACC05}, we infer convexity of $\mathcal{J}$ above $\overline{V}=|E^m_{h_\Om}|$ for $n$-dimensional, bounded, $C^{1,1}$ regular, convex sets $\Om$. If one could extend the arguments of~\cite{ACC05}, which exploit the maximum principle and Korevaar's comparison principle~\cite{Kor83a, Kor83b}, to $(n-1)R^{-1}_\Om\le \k<h_\Om$, then one would immediately get the convexity of  $\mathcal{J}$ for all $V\ge \omega_n R^n_\Om$ and the global convexity of its $n(n-1)^{-1}$ power by following our same proofs.

It is interesting to compare these convexity properties, to those of the \emph{relative} isoperimetric profile
\[
\mathcal{J}_{\text{rel}}(V):=\inf\left\{\,P(E; \Omega)\,:\, |E|=V,\, E\subset \Omega \,\right\}\,,
\]
that have been well studied: the first result in this direction were obtained for bounded, $C^{2,\alpha}$ regular, convex sets $\Omega$, and it was shown that the profile is concave~\cite{SZ99} and so it is its $n(n-1)^{-1}$ power~\cite{Kuw03}; these have been later extended to bounded, convex bodies without any further regularity assumption on their boundaries~\cite[Section~6]{Mil09}, see also~\cite{RV15}; then to arbitrary, unbounded, convex bodies~\cite{LRV18}.

The paper is structured as follows. In \cref{sec:main_results} we give the definition of the class of sets $\Omega$ we will consider along with the main results. In \cref{sec:shape_min} we prove several properties of minimizers of~\eqref{eq:pmc_mod} and then prove the geometric characterization of its minimizers. \cref{sec:isop_prof} deals with the isoperimetric profile. Finally, \cref{sec:examples} completes the paper with few explicit examples.

\section{Main results}\label{sec:main_results}

In this section we state and comment the main results of the paper. We start by the following definition, first introduced in~\cite{LNS17}.

\begin{defin}\label{def:no_necks}
A set $\Om$ has \emph{no necks of radius $r$}, with $r\in(0, R_\Om]$ if the following condition holds. If $B_r(x_0)$ and $B_r(x_1)$ are two balls of radius $r$ contained in $\Om$, then there exists a continuous curve $\g\colon [0,1]\to \Om$ such that 
\[
\g(0)=x_0, \qquad \g(1)=x_1, \qquad B_r(\g(t)) \subset \Om,\quad \forall t\in [0,1].
\]
\end{defin}

Before stating our main results, we need to recollect the following proposition which introduces some of the notation that we will need later on. For the sake of completeness, we recall that a Jordan curve is the image of a continuous and injective map $\Phi:\mathbb{S}^1\to \mathbb{R}^2$ and a Jordan domain is the domain bounded by such a curve, which is well defined thanks to the Jordan--Schoenflies Theorem. We also recall that for $r\le R_\Om$, the set $\Om^r$ is the (closed) inner parallel set of $\Om$ at distance $r$, i.e.,
\[
\Om^r := \{\,x\in \Omega\,:\, \dist(x; \de \Omega) \ge r \,\}\,.
\]
The reach of a closed set $A$, which was introduced in the seminal paper~\cite{FedererCM} (see also the recent book~\cite{RZ19book}) is defined as
\[
\reach(A) := \sup \{\, r : \forall x \in A\oplus B_r\,, \, x \text{ has a unique projection onto } A\,\}\,,
\]
where $\oplus$ denotes the \emph{Minkowski sum}, and we use the notation $B_r = B_r(0)$.

In the next proposition we collect various results from~\cite{LS20} (in particular, see Proposition~2.1, Remark~4.2, Lemma~5.3, and Remark~5.4).
\begin{prop}\label{prop:struttura_diff}
Let $\Om$ be a Jordan domain with no necks of radius $r$. The following properties hold:
\begin{itemize}
\item[(a)] if $\Om^r$ is nonempty but has empty interior, then either it consists of a single point or there exists an embedding $\g\colon[0,1]\to \R^2$ of class $C^{1,1}$, with curvature bounded by $r^{-1}$, such that $\g([0,1]) = \Om^r$;
\item[(b)] if $\interior(\Om^r) \neq \emptyset$, then there exist two (possibly empty) families $\G^1_r$ and $\G^2_r$ of embedded curves contained in $\Om^r$ with the following properties. For each $i=1,2$ and each $\g\in \G^i_r$, 
\begin{itemize}
\item[(i)] $\g\colon[0,1]\to \Om^r$ is nonconstant and of class $C^{1,1}$, with curvature bounded by $r^{-1}$;
\item[(ii)] if $i=1$, then $\overline{\interior(\Om^r)}\cap \g = \{\g(0)\}$;
\item[(iii)] if $i=2$, then $\overline{\interior(\Om^r)}\cap \g = \{\g(0), \g(1)\}$;
\item[(iv)] $\G^1_r$ is a finite collection of curves;
\item[(v)] we have
\[
\Om^r \setminus \overline{\interior(\Om^r)} = \bigcup_{\g\in \G^1_r} \g \left((0,1]\right) \cup \bigcup_{\g \in \G^2_r} \g \left((0,1)\right).
\]
\end{itemize}
Moreover, for every $\theta: \G^1_r\to [0,1]$ the compact set
\[
C_\theta = \overline{\interior(\Om^r)} \cup \bigcup_{\g \in \G^2_r} \g([0,1]) \cup \bigcup_{\g \in \G^1_r} \g([0,\theta(\g)])\,,
\]
is simply connected and such that $\reach(C_\theta)\ge r$. By $C_0$ and $C_1$ we shall denote the sets obtained with the choices $\theta(\g) \equiv 0,1$ for every $\g\in \G^{1}_{r}$.

\noindent
Finally, if $\Om$ has no necks of radius $r$ for all $r\in [\bar r, \bar r+\e]$, and for given $\bar r>0$ and $\e>0$, then $\G^2_r =\emptyset$ for all such $r$.
\end{itemize}
\end{prop}

To have a better picture of the situation described in \cref{prop:struttura_diff}, we refer to \cref{fig:g1_g2}. Loosely speaking, curves in $\G^1_r$ correspond to the presence of ``tendrils'' of width $r$ in $\Omega$, while curves in $\G^2_r$ correspond to ``handles'' of width $r$ between different connected components of $\interior(\Om^r)$. For the sake of completeness, we notice that the last part of the previous statement, corresponding to~\cite[Remark~4.2]{LS20}, is written here in a local form, while it was originally stated under the global assumption of no necks for all $r\in [0,\bar r]$. The proof of this part is exactly the same as the one outlined in that remark.

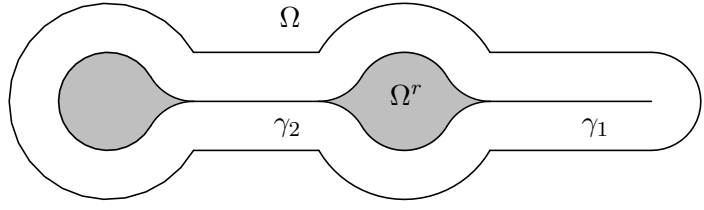
\begin{figure}
\begin{tikzpicture}[line cap=round,line join=round,>=triangle 45,x=1.0cm,y=1.0cm, scale=.65]
\begin{scope}[yscale=1, xscale=-1]
\clip(-6.2,-2.2) rectangle (8.2,2.2);
\fill[line width=0.pt,color=cqcqcq,fill=cqcqcq,fill opacity=1.0] (-1.7320508075688772,0.) -- (-0.866025403784438,0.5) -- (-0.8597787687708186,-0.5106666904850296) -- cycle;
\fill[line width=0.pt,color=cqcqcq,fill=cqcqcq,fill opacity=1.0] (1.7320508075688772,0.) -- (0.8660254037844382,0.5) -- (0.8597787687708185,-0.5106666904850297) -- cycle;
\fill[line width=0.pt,color=cqcqcq,fill=cqcqcq,fill opacity=1.0] (4.267949192431123,0.) -- (5.133974596215562,0.5) -- (5.140221231229182,-0.5106666904850299) -- cycle;
\draw [line width=0.6pt,color=cqcqcq,fill=cqcqcq,fill opacity=1.0] (6.,0.) circle (1.cm);
\draw [line width=0.6pt,color=ffffff,fill=ffffff,fill opacity=1.0] (4.267949192431123,1.) circle (1.cm);
\draw [line width=0.6pt,color=cqcqcq,fill=cqcqcq,fill opacity=1.0] (0.,0.) circle (1.cm);
\draw [line width=0.6pt,color=ffffff,fill=ffffff,fill opacity=1.0] (1.732050807568877,1.) circle (1.cm);
\draw [line width=0.6pt,color=ffffff,fill=ffffff,fill opacity=1.0] (-1.7320508075688767,1.) circle (1.cm);
\draw [line width=0.6pt,color=ffffff,fill=ffffff,fill opacity=1.0] (1.732050807568877,-1.) circle (1.cm);
\draw [line width=0.6pt,color=ffffff,fill=ffffff,fill opacity=1.0] (-1.7320508075688767,-1.) circle (1.cm);
\draw [line width=0.6pt,color=ffffff,fill=ffffff,fill opacity=1.0] (4.267949192431123,-1.) circle (1.cm);
\draw [line width=0.6pt] (-5.,1.)-- (-1.7320508075688772,1.);
\draw [line width=0.6pt] (1.7320508075688774,1.)-- (4.267949192431123,1.);
\draw [line width=0.6pt] (-5.,-1.)-- (-1.7320508075688772,-1.);
\draw [line width=0.6pt] (1.7320508075688774,-1.)-- (4.267949192431123,-1.);
\draw [shift={(-5.,0.)},line width=0.6pt]  plot[domain=1.5707963267948966:4.71238898038469,variable=\t]({1.*1.*cos(\t r)+0.*1.*sin(\t r)},{0.*1.*cos(\t r)+1.*1.*sin(\t r)});
\draw [shift={(0.,0.)},line width=0.6pt]  plot[domain=0.5235987755982988:2.6179938779914944,variable=\t]({1.*2.*cos(\t r)+0.*2.*sin(\t r)},{0.*2.*cos(\t r)+1.*2.*sin(\t r)});
\draw [shift={(0.,0.)},line width=0.6pt]  plot[domain=3.665191429188092:5.759586531581287,variable=\t]({1.*2.*cos(\t r)+0.*2.*sin(\t r)},{0.*2.*cos(\t r)+1.*2.*sin(\t r)});
\draw [shift={(-1.7320508075688772,1.)},line width=0.6pt]  plot[domain=4.71238898038469:5.759586531581288,variable=\t]({1.*1.*cos(\t r)+0.*1.*sin(\t r)},{0.*1.*cos(\t r)+1.*1.*sin(\t r)});
\draw [line width=0.6pt] (-5.,0.)-- (-1.7320508075688772,0.);
\draw [shift={(-1.732050807568877,-1.)},line width=0.6pt]  plot[domain=4.71238898038469:5.759586531581288,variable=\t]({1.*1.*cos(\t r)+0.*1.*sin(\t r)},{0.*1.*cos(\t r)+-1.*1.*sin(\t r)});
\draw [shift={(1.7320508075688774,1.)},line width=0.6pt]  plot[domain=4.71238898038469:5.759586531581288,variable=\t]({-1.*1.*cos(\t r)+0.*1.*sin(\t r)},{0.*1.*cos(\t r)+1.*1.*sin(\t r)});
\draw [shift={(1.7320508075688767,-1.)},line width=0.6pt]  plot[domain=4.71238898038469:5.759586531581288,variable=\t]({-1.*1.*cos(\t r)+0.*1.*sin(\t r)},{0.*1.*cos(\t r)+-1.*1.*sin(\t r)});
\draw [shift={(4.267949192431123,1.)},line width=0.6pt]  plot[domain=4.71238898038469:5.759586531581288,variable=\t]({1.*1.*cos(\t r)+0.*1.*sin(\t r)},{0.*1.*cos(\t r)+1.*1.*sin(\t r)});
\draw [shift={(4.267949192431123,-1.)},line width=0.6pt]  plot[domain=4.71238898038469:5.759586531581288,variable=\t]({1.*1.*cos(\t r)+0.*1.*sin(\t r)},{0.*1.*cos(\t r)+-1.*1.*sin(\t r)});
\draw [line width=0.6pt] (1.7320508075688772,0.)-- (4.267949192431123,0.);
\draw [shift={(6.,0.)},line width=0.6pt]  plot[domain=-2.617993877991494:2.617993877991494,variable=\t]({1.*1.*cos(\t r)+0.*1.*sin(\t r)},{0.*1.*cos(\t r)+1.*1.*sin(\t r)});
\draw [shift={(0.,0.)},line width=0.6pt]  plot[domain=3.6651914291880923:5.759586531581287,variable=\t]({1.*1.*cos(\t r)+0.*1.*sin(\t r)},{0.*1.*cos(\t r)+1.*1.*sin(\t r)});
\draw [shift={(0.,0.)},line width=0.6pt]  plot[domain=0.5235987755982994:2.617993877991494,variable=\t]({1.*1.*cos(\t r)+0.*1.*sin(\t r)},{0.*1.*cos(\t r)+1.*1.*sin(\t r)});
\draw [shift={(6.,0.)},line width=0.6pt]  plot[domain=-2.617993877991494:2.617993877991494,variable=\t]({1.*2.*cos(\t r)+0.*2.*sin(\t r)},{0.*2.*cos(\t r)+1.*2.*sin(\t r)});
\draw (2.85396543022974,-.15) node[anchor=north west] {$\gamma_2$};
\draw (-3.3815835625861905,-.15) node[anchor=north west] {$\gamma_1$};
\draw (2.7340510265217413,2.1279954111917614) node[anchor=north west] {$\Omega$};
\draw (0.5,0.5391295620607859) node[anchor=north west] {$\Omega^r$};
\end{scope}
\end{tikzpicture}
\caption{A set with a curve in $\G^1_r$ and one in $\G^2_r$.}\label{fig:g1_g2}
\end{figure}

\medskip

The next result provides a precise, geometric characterization of all minimizers of $\mathcal{F}_\k$. 

\begin{thm}\label{thm:main_shape_min}
Let $\Om$ be a Jordan domain with $|\de \Om|=0$ and let $\k \ge R_\Om^{-1}$ be fixed. Assume $\Om$ has no necks of radius $r=\k^{-1}$. Let $E_{\k}$ be a minimizer of problem~\eqref{eq:pmc_mod}, that is, of the functional ${\mathcal F}_{\k}$ restricted to the class $\mathcal{C}(\kappa)$. Then, with reference to the notation introduced in \cref{prop:struttura_diff}, the following properties hold:
\begin{itemize}
\item[(i)] if $r< R_{\Om}$ and $\G^1_{r}\neq \emptyset$, then there exists $\theta:\G^1_{r}\to [0,1]$ such that 
\[
E_{\k} = C_{\theta}\oplus B_{r}\,,
\]
and in particular 
\[
E_{\k}^{m} = C_{0}\oplus B_{r},\quad E_{\k}^{M} = C_{1}\oplus B_{r} = \Om^{r}\oplus B_{r}
\]
are, respectively, the unique minimal and maximal minimizers of ${\mathcal F}_{\k}$;

\item[(ii)] if $r< R_{\Om}$ and $\G^1_{r}= \emptyset$, then $E_{\k}$ is the unique minimizer of ${\mathcal F}_{\k}$, given by
\[
E_{\k} = \Om^{r} \oplus B_{r};
\]

\item[(iii)] if $r=R_\Om$, then $\Om^{r}$ is a closed curve of class $C^{1,1}$ (possibly reduced to a point) and there exists a connected subset $K\subset \Om^{r}$ such that 
\[
E_{\k} = K \oplus B_{r}\,.
\]
Moreover any ball of radius $r$ centered on $\Om^{r}$ is a minimal minimizer, while $E_{\k}^{M} = \Om^{r}\oplus B_{r}$ is the unique maximal minimizer.
\end{itemize}
\end{thm}

Apart from the technical assumption $|\de\Om|=0$, the other hypotheses of \cref{thm:main_shape_min} are sharp, as showed by suitable examples (see~\cite{LNS17, LS18b}). Moreover, the minimizers appearing in the theorem are also solutions of the isoperimetric problem~\eqref{eq:isop_prof} (with $V$ equal to their own volume). 

The next result shows the converse, that is, any solution of the isoperimetric problem~\eqref{eq:isop_prof}, for a prescribed volume $V\ge \pi R_{\Om}^{2}$, is also a minimizer of $\mathcal{F}_\k$ for some $\k\ge R_{\Om}^{-1}$. As a consequence, all isoperimetric solutions are geometrically characterized as in \cref{thm:main_shape_min}.

\begin{thm}\label{thm:m_iso}
Let $\Om$ be a Jordan domain with $|\de \Om|=0$ and without necks of radius $r$, for all $r\in(0, R_\Om]$. Then, for all $V\in [\pi R_{\Om}^{2},|\Omega|)$ there exists a unique $\kappa\in [R_\Om^{-1}, +\infty)$ such that a set $E$ of volume $V$ is isoperimetric if and only if it minimizes $\mathcal{F}_\kappa$ in $\mathcal{C}(\k)$.
\end{thm}

This theorem has few consequences. First, it represents an extension of~\cite[Theorem~3.32]{SZ97}, where the authors classify isoperimetric sets within $2$d convex sets, to the much richer class of Jordan domains without necks (for instance, the Koch snowflake belongs to this class, see~\cite[Section 6]{LNS17}). Second, it supports some algorithmic procedures for the construction of isoperimetric sets and for the computation of the isoperimetric profile, see~\cite{ZdFS20}, where the authors had in mind the political phenomenon of gerrymandering, which can be discussed using isoperimetric arguments, see also~\cite{FLSS19, SS19}. Indeed, in~\cite{ZdFS20} the authors numerically computed the isoperimetric profile of a Jordan domain $\Om$ by implicitly assuming the characterization of minimizers that we have now completely proved here. We remark that in such numerical applications, the way to go is to consider the cut locus $\mathcal{C}$ (also known as medial axis), i.e., the set of points in $\Om$ where the distance function from the boundary is not differentiable. This happens precisely because such points have more than one projection on the boundary. To every $x\in \mathcal{C}$ one can associate the so-called \emph{medial axis transform} $f(x)=\dist(x; \de \Om)$, which simply evaluates the distance from the boundary. Then, one notices that $f$ attains its maximum on the set $\Om^{R_\Om}$. As the cut locus has a tree structure, see~\cite[Section~3]{LNS17}, one can prove that the function $f$ decreases while moving away from $\Om^{R_\Om}$, and it is locally constant at $x$ only if the set $\Om$ is such that $\G^1_{f(x)}\neq \emptyset$. Then, in order to build isoperimetric sets, one is lead to consider the union
\[
\bigcup_{\substack{x\in \mathcal{C}\\ f(x)>r}} B_r(x)\,.
\]
Whenever $\Om$ is such that $\G^1_r=\emptyset$ for all $r\le R_\Om$, \cref{thm:main_shape_min} and~\cref{thm:m_iso} guarantee that the above procedure yields all isoperimetric sets.

Third and finally, it allows us to prove the following result about convexity properties of the isoperimetric profile $\mathcal{J}$ (see \cref{sec:convexity}). 

\begin{thm}\label{thm:convexity_J}
Let $\Om$ be a Jordan domain with $|\de \Om|=0$ and without necks of radius $r$, for all $r\in(0, R_\Om]$. Then, the isoperimetric profile $\mathcal{J}$ is convex in $[\pi R^2_\Om, |\Om|]$, while $\mathcal{J}^2$ is convex in $[0,|\Om|]$.
\end{thm}

The proof of convexity of $\mathcal{J}$ essentially relies on the existence of a threshold $\overline{V}$ such that minimizers of $\mathcal{J}$ also minimize $\mathcal{F}_\k$ for a suitable $\k$. It is worth noticing that it does not rely on the precise shape of the minimizer. Indeed, using the results of~\cite{ACC05}, we can prove as well that the isoperimetric profile for $V\ge |E^m_{h_\Om}|$ is convex for $n$-dimensional, convex, $C^{1,1}$ regular sets $\Om$, see \cref{sec:dim_n}.

\section{Shape of minimizers}\label{sec:shape_min}

This section is dedicated to the proof of \cref{thm:main_shape_min}, which is contained in \cref{ssec:proof_shape}. We recall that from now on (unless explicitly stated) $\Omega$ denotes an open bounded set in $\R^{2}$. Prior to the proof, we need to discuss some properties of minimizers {\color{blue} of~\eqref{eq:pmc_mod}}, see \cref{ssec:properties}. Moreover, when $\Om$ is a Jordan domain, one has some additional properties, see \cref{ssec:jordan_domains}. These results are generalizations and/or adaptations of those already known in the case $\k\ge h_\Om$. Whenever the adaptation is straightforward, no proof is given and a reference is provided. 

\subsection{Properties of minimizers}\label{ssec:properties}

First thing we need to prove is that there exist minimizers of $\mathcal{F}_\k$ when $R_\Om^{-1}\le \k< h_\Om$. One easily checks it by taking a minimizing sequence which, up to subsequences, is shown to converge in the $BV$ topology. By lower semicontinuity of the perimeter, the limit set is a minimizer, provided it belongs to $\mathcal{C}(\k)$. Details are given in the following proposition.

\begin{prop}\label{prop:existence}
Let $\k\ge R_\Om^{-1}$. There exist non trivial minimizers of \eqref{eq:pmc_mod}, that is, of $\mathcal{F}_\kappa$ restricted to the class $\mathcal{C}(\kappa)$.
\end{prop}

\begin{proof}
If $\k\ge h_\Om$ this is well known. Let now $R_\Om^{-1}\le \k<h_\Om$, and notice that this choice ensures that the class of competitors is nonempty. Moreover, the functional is clearly bounded from below by $-\k|\Om|$, thus we can pick $\{E_h\}_h$ a minimizing sequence in $\mathcal{C}(\k)$. Without loss of generality we may assume that
\[
P(E_h) -\k |E_h| \le \inf_{\mathcal{C}(\k)}\{\mathcal{F}_\k[E]\} + 1\,,
\]
and thus
\[
P(E_h) \le \inf_{\mathcal{C}(\k)}\{\mathcal{F}_\k[E]\} + 1 + \k|\Om| \le \text{const}\,,
\]
as $\Om$ is bounded. Therefore, up to subsequences, $E_h$ converges in the $L^{1}$ topology to a limit set $E$. As $|E_h|\ge \pi \k^{-2}$ for all $h$, by taking the limit as $h\to \infty$ we infer $|E|\ge \pi \k^{-2}$. This shows that $E$ belongs to $\mathcal{C}(\k)$, hence the lower semicontinuity of the perimeter yields the fact that $E$ is indeed a nontrivial minimizer of $\mathcal{F}_{\k}$ in $\mathcal{C}(\k)$.
\end{proof}

There are several, well-established properties of non trivial minimizers of $\mathcal{F}_\k$, for $\k\ge h_\Om$, which hold the same for $\k\ge R_\Om^{-1}$. We recall them below.

\begin{prop}\label{prop:properties}
Let $E_\k$ be a minimizer of \eqref{eq:pmc_mod}, that is, of $\cF_\k$ restricted to $\mathcal{C}(\kappa)$. Then, the following statements hold true:
\begin{itemize}
\item[(i)] $\de E_\k \cap \Om$ is analytic and coincides with a countable union of circular arcs of curvature $\k$, with endpoints belonging to $\de\Om$;
\item[(ii)] the length of any connected component of $\de E_\k \cap \Om$ cannot exceed $\pi \k^{-1}$;
\item[(iii)] for $\Omega$ with locally finite perimeter, if $x\in \de E_\k \cap \de^* \Om$, then $x\in \de^*E_\k$ and $\nu_\Om(x) = \nu_{E_\k}(x)$.
\end{itemize}
\end{prop}

Point~(i) for $\k > h_\Om$ comes from the regularity of perimeter minimizers, and the condition of the curvature directly from writing down the first variation of the functional, refer for instance to~\cite[Section~17.3]{Mag12book}. For $\k\le h_\Om$ one has to consider two cases in view of the definition of $\mathcal{C}(\kappa)$. Either $|E_\k|> \pi \k^{-2}$, and then one is allowed to make variations changing the volume both from above and below, thus obtaining the same result. Or the equality $|E_\k|= \pi \k^{-2}$ holds, and then $\Omega$ contains a ball of curvature $\k$ because it contains a ball of radius $R_{\Omega}$. Therefore, by the isoperimetric inequality, all such balls are the only minimizers of $\mathcal{F}_\k$. Point~(ii) can be proved exactly the same as in~\cite[Lemma~2.11]{LP16}. Point~(iii) comes from regularity properties of $(\Lambda, r_0)$-minimizers; a proof for Lipschitz $\Om$ is available in~\cite{GMT81}, while for sets $\Om$ with just locally finite perimeter we refer to~\cite[Theorem~3.5]{LS18a}. 

\begin{rem}\label{rem:k_piu_piccoli}
There is nothing preventing us from minimizing $\mathcal{F}_\kappa$ in the class $\{E\subset \Omega\,, \text{Borel, s.t. } |E|\ge \pi\k^{-2}\,\}$, for $\sqrt{\pi}|\Om|^{-\frac 12}\le \k < R_{\Om}^{-1}$ (this choice of $\k$ ensures that the class is nonempty). The issue here is that property~(i) can be no more guaranteed. Indeed, if a minimizer $E_\k$ has volume exactly equal to $\pi\k^{-2}$, we are only allowed to make outer variations, obtaining that the curvature of $\de E_{\k}\cap \Om$ is not smaller than $\k$. In other words, the lack of a ball contained in $\Omega$ with volume at least $\pi \k^{-2}$ prevents us from recovering an equality on the curvature, leaving us just with an inequality. A similar remark is valid in the general, $n$-dimensional case.
\end{rem}

For $\k\ge h_\Om$ it is well-known that the class of minimizers is closed with respect to countable union and intersections, see for instance the first part of the proof of~\cite[Proposition~3.2]{LS20} or~\cite[Lemma~2.2 and Remark~4.2]{CCN10}. As $\k$ drops below $h_\Om$ this is still true, provided that the intersection is still a viable competitor, as we show in the next lemma.

\begin{prop}\label{prop:closed_class}
Let $\k \in [R_\Om^{-1}, h_\Om)$, and let $E_\k, F_\k$ be minimizers of $\mathcal{F}_\k$. Assume that $|E_\k\cap F_\k|\ge \pi \kappa^{-2}$. Then, the union $E_\k \cup F_\k$ and the intersection $E_\k \cap F_\k$ are minimizers of $\mathcal{F}_\k$.
\end{prop}

\begin{proof}
By minimality we have
\begin{align*}
P(E_\k)+P(F_\k) - 2\min \cF_\k &= \k(|E_\k| +|F_\k|) = \k|E_\k\cup F_\k|+ \k|E_\k\cap F_\k|\\
&\le P(E_\k \cup F_\k) + P(E_\k \cap F_\k) - 2\min \cF_\k\,.
\end{align*}
By the well-known inequality (see for instance~\cite[Lemma~12.22]{Mag12book})
\[
P(E_\k \cup F_\k) + P(E_\k \cap F_\k) \le P(E_\k) +P(F_\k)\,,
\]
and provided that $|E_\k\cap F_\k|\ge \pi \k^{-2}$, we obtain
\begin{align*}
P(E_\k)+P(F_\k) - 2\min \cF_\k & = \k|E_\k\cup F_\k|+ \k|E_\k\cap F_\k|\\
&\le P(E_\k \cup F_\k) + P(E_\k \cap F_\k)  - 2\min \cF_\k\\
&\le P(E_\k)+P(F_\k) - 2\min \cF_\k\,,
\end{align*}
therefore the two last inequalities must be equalities. But this can happen if and only if
\[
P(E_\k \cup F_\k) - \k|E_\k\cup F_\k| = P(E_\k \cap F_\k) - \k|E_\k\cap F_\k| = \min \cF_\k\,.\qedhere
\]
\end{proof}

We stress the necessity of requiring that the measure of the intersection is big enough, otherwise one can easily produce counterexamples, as shown in \cref{fig:not_minimizer}. The choice $\k=R_\Om^{-1}$ is not restrictive, as one can find counterexamples for general $\k\in[R_\Om^{-1}, h_\Om)$ by considering a suitable ``balanced dumbbell'', namely two identical squares linked by a very thin corridor, as in \cref{fig:balanced_dumbbell}. 

\begin{figure}[t]
\makebox[\linewidth][c]{
\begin{subfigure}[b]{.25\linewidth}
\begin{tikzpicture}[thick]
\clip (-1.1,-1) rectangle (2.1,1);
\draw[thin] (-1,-1) rectangle (2,1);
\fill[black!20!white] (1,0) circle (1cm);
\draw[dashed, thin](0,0) circle (1cm);
\draw[thin] (1,0) circle (1cm);
\end{tikzpicture}
\caption{A minimizer\label{fig:label1}}
\end{subfigure}
\quad
\begin{subfigure}[b]{.25\linewidth}
\begin{tikzpicture}[thick]
\clip (-1.1,-1) rectangle (2.1,1);
\draw[thin](-1,-1) rectangle (2,1);
\fill[black!20!white](0,0) circle (1cm);
\draw[dashed, thin](1,0) circle (1cm);
\draw[thin] (0,0) circle (1cm);
\end{tikzpicture}
\caption{A minimizer\label{fig:label2}}
\end{subfigure}
\quad
\begin{subfigure}[b]{.25\linewidth}
\begin{tikzpicture}[thick]
\clip (-1.1,-1) rectangle (2.1,1);
\draw[thin] (-1,-1) rectangle (2,1);
\draw[dashed, thin](0,0) circle (1cm);
\draw[dashed, thin](1,0) circle (1cm);
\clip(0,0) circle (1cm);
\fill[black!20!white] (1,0) circle (1cm);
\end{tikzpicture}
\caption{The intersection\label{fig:label3}}
\end{subfigure}
\quad
\begin{subfigure}[b]{.25\linewidth}
\begin{tikzpicture}[thick]
\clip (-1.1,-1) rectangle (2.1,1);
\draw[thin] (-1,-1) rectangle (2,1);
\fill[black!20!white] (1,0) circle (1cm);
\fill[black!20!white] (0,0) circle (1cm);
\draw[dashed, thin](0,0) circle (1cm);
\draw[dashed, thin](1,0) circle (1cm);
\end{tikzpicture}
\caption{The union\label{fig:label4}}
\end{subfigure}
}
\caption{On the left two minimizers for $\k=R_\Om^{-1}$; on the right their intersection and union which are not minimizers.\label{fig:not_minimizer}}
\end{figure}
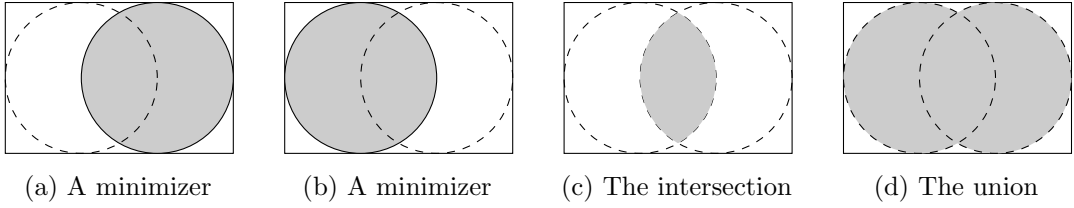

As the class of minimizers is not closed under (countable) unions or intersections when $R_\Om^{-1}\le \k<h_\Om$, one cannot directly define the maximal minimizer as the union of all minimizers, or the minimal minimizer as the intersection of all minimizers (for $\k> h_{\Om}$) as in ~\cite[Definition~2.2, Proposition~3.2 and Remark~3.3]{LS20} or~\cite[Definition~2.1, Lemma~2.2 and Lemma~2.5]{CCN10}. Nevertheless, we can give an alternative definition in terms of maximal/minimal volume.

\begin{defin}\label{def:min_max}
A minimal minimizer of $\mathcal{F}_\kappa$ is a set $E^{m}_\k$ belonging to
\[
\argmin \{\,|E_\kappa|\,:\, E_\kappa \text{ is a minimizer of $\mathcal{F}_\kappa$}\,\}\,.
\]
Similarly, a maximal minimizer is a set $E^M_\kappa$ belonging to
\[
\argmax \{\,|E_\kappa|\,:\, E_\kappa \text{ is a minimizer of $\mathcal{F}_\kappa$}\,\}\,.
\]
\end{defin}

\begin{prop}
There exist both minimal and maximal minimizers of $\mathcal{F}_\kappa$.
\end{prop}

\begin{proof}
Take any extremizing sequence of minimizers. Up to subsequences, it converges to a limit set $E$, such that its volume is the infimum (or the supremum) of the volumes. As in the proof of \cref{prop:existence}, one sees that $E$ is a minimizer.
\end{proof}

When dealing with $R_\Om^{-1}\le \k<h_\Om$, we remark that in contrast with the case $\k\ge h_\Om$ one might have multiple maximal minimizers, and in contrast with $\k>h_\Om$ multiple minimal minimizers. Consider for instance a balanced dumbbell, depicted in \cref{fig:balanced_dumbbell}. If the handle is thin enough, for all $\k\in [R_\Om^{-1}, h_\Om)$, there are exactly two minimizers, corresponding to the two components shaded in gray in the figure. Both are at the same time maximal and minimal minimizers, while their union and intersection are not minimizers (indeed, the value of $\mathcal{F}_\kappa$ on the union is twice the \emph{positive} infimum of $\mathcal{F}_\kappa$, while the intersection is empty and hence not in ${\mathcal C}_{\k}$).

\begin{figure}
\begin{tikzpicture}[thick, scale=0.17]
\fill[black!20!white] (14,9) circle (5cm);
\fill[black!20!white] (5,5) circle (5cm);
\fill[black!20!white] (5,13) circle (5cm);
\fill[black!20!white] (13,13) circle (5cm);
\fill[black!20!white] (13,5) circle (5cm);
\fill[black!20!white] (0,5)--(0,13)--(5,18)--(13,18)--(18,13)--(18,5)--(13,0)--(5,0)--(0,5);
\begin{scope}[xscale=-1, xshift=-42cm]
\fill[black!20!white] (14,9) circle (5cm);
\fill[black!20!white] (5,5) circle (5cm);
\fill[black!20!white] (5,13) circle (5cm);
\fill[black!20!white] (13,13) circle (5cm);
\fill[black!20!white] (13,5) circle (5cm);
\fill[black!20!white] (0,5)--(0,13)--(5,18)--(13,18)--(18,13)--(18,5)--(13,0)--(5,0)--(0,5);
\end{scope}
\draw[thin] (0,0)--(0,18)--(18,18)--(18,12)--(24,12)--(24,18)--(42,18)--(42,0)--(24,0)--(24,6)--(18,6)--(18,0)--(0,0);
\end{tikzpicture}
\caption{For values $\k\in[R_\Om^{-1}, h_\Om)$, a balanced dumbbell has two minimizers. Both are maximal and minimal at the same time.\label{fig:balanced_dumbbell}}
\end{figure}
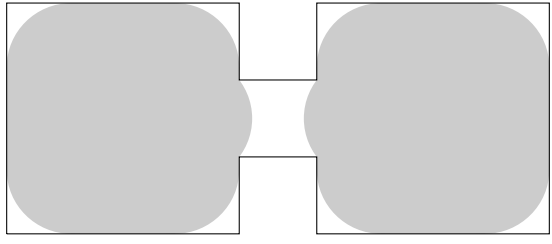

Finally, we provide an extended version of the \emph{rolling ball lemma}~\cite[Lemma~2.12]{LP16}, which also extends~\cite[Lemma~1.7]{LNS17}. This lemma, already known in the case $\k = h_{\Om}$, holds for a general $\k$ and its proof is an easy adaptation of the original one. Before stating the lemma, we introduce some needed terminology.  Given two balls $B_{r}(x_{0}),B_{r}(x_{1})\subset \Om$, with same radius but possibly different centers, we say that $B_{r}(x_{0})$ can be rolled onto $B_{r}(x_{1})$ if there exists a continuous curve $\g\colon[0,1]\to \Omega$ with $\g(0)=x_0$, $\g(1)=x_1$ and such that $B_r(\g(t))\subset \Om$ for all $t\in[0,1]$ (such $\g$ will be called a \textit{rolling curve}).

\begin{lem}[Rolling ball - extended version]\label{lem:rollingball}
Let $\k = r^{-1}\ge R_\Om^{-1}$ be fixed, and let $E_{\k}$ be a minimizer of $\cF_\k$. Then the following properties hold:
\begin{itemize}
\item[(i)] if $E_\k$ contains a ball $B_r(x_0)$, then it contains any ball $B_{r}(x_{1})$ with $x_{1}\in \interior(\Omega^{r})$, and such that $B_r(x_0)$ can be rolled onto $B_r(x_{1})$;

\item[(ii)] if $E_{\k}$ contains two balls $B_r(x_0)$ and $B_r(x_1)$ that can be rolled onto each other, then it contains all balls of radius $r$ centered on the points of the rolling curve;

\item[(iii)] if $E_{\k}$ is a maximal minimizer of $\cF_\k$ such that $B_r(x_0)\subset E_{\k}$ for some $x_{0}$, then $E_{\k}$ contains any other ball $B_r(x_1)$, which $B_r(x_0)$ can be rolled onto. 
\end{itemize}
\end{lem}

\begin{proof}
Point~(iii) was first proved in~\cite[Lemma~2.12]{LP16} and later refined in~\cite[Lemma~1.7]{LNS17}. 

As for point~(i), one can argue by contradiction assuming that $B_{r}(x_{1})$ is not contained in $E_{\k}$. Then, fix $\g \colon[0,1]\to \Omega$ with $\g(0)=x_0$, $\g(1)=x_1$ and such that $B_r(\g(t))\subset \Om$ for all $t\in[0,1]$, and assume that $\g$ is parametrized by a multiple of the arc-length. Denote by $\alpha^{+}(t)$ the half-circle made by the points of the form $\g(t) + r \nu$, with $\nu$ such that $|\nu|=1$ and $\nu \cdot \dot{\g}(t)>0$. Denote as $t^{*}$ the supremum of $t\in [0,1]$ such that $B_r(\g(s))\subset E_{\k}$ for all $s\in [0,t]$. Clearly we have $t^{*}<1$. Arguing as in \cite[Lemma~2.12]{LP16}, we infer that $\alpha^{+}(t^{*})$ coincides with a connected component of $\de E_{\k}\cap \Om$, and the set $\widetilde{E}_{\k} = E_{\k}\cup \bigcup_{s\in [t^{*},1]} B_r(\g(s))$ turns out to be a minimizer as well. Thus by \cref{prop:properties}~(i) and~(ii), a connected component of $\de \widetilde{E}_{\k}\cap \Om$ not larger than a half-circle should be contained in $\de B_{r}(x_{1})$, which is a compact subset of $\Om$, hence its endpoints should belong to $\Om\cap \de\Om = \emptyset$, a contradiction. 

The proof of point (ii) is achieved through a similar argument to the one used for proving point (i). Let $t^{*}$ denote the supremum of $t\in [0,1]$ such that $B_r(\g(s))\subset E_{\k}$ for all $s\in [0,t]$, and similarly let $t_*$ be the infimum of $t\in [0,1]$ such that $B_r(\g(s))\subset E_{\k}$ for all $s\in [t,1]$. Assume by contradiction that $t^{*}<t_*$. Denoting by $\alpha^{-}(t)$ the half-circle whose points are of the form $\g(t) + r \nu$, with $\nu$ such that $|\nu|=1$ and $\nu \cdot \dot{\g}(t)<0$, we have that both $\alpha^{+}(t^{*})$ and $\alpha^{-}(t_*)$ coincide with two distinct connected components of $\de E_{\k}\cap \Om$. Finally, by rolling $B_{r}(\g(t^{*}))$ towards $B_{r}(\g(t_*))$ we could construct a minimizer that would exhibit either a non-admissible singularity of cuspidal type on $\de\Om$, see \cref{fig:cuspidal}, or a non-admissible singularity of ``rounded X'' type for $\de E\cap \Om$, see \cref{fig:x_rounded}. Both cannot happen, as ``cutting the singularity'' would produce a competitor with smaller perimeter and greater volume.
\begin{figure}[t]
\makebox[\linewidth][c]{
\begin{subfigure}[b]{.35\linewidth}
\centering
\begin{tikzpicture}[thick]
\begin{scope}
    \clip (0,-0.1) rectangle (2,2);
    \fill[black!20!white]  (1,0) circle(1);
    \draw[dashed]  (1,0) circle(1);
\end{scope}
\begin{scope}
    \clip (-2,2) rectangle (0,-0.1);
    \fill[black!20!white]  (-1,0) circle(1);
    \draw[dashed]  (-1,0) circle(1);
\end{scope}
\begin{scope}
    \clip (-2,0) rectangle (2,2.5);
    \draw[black]  (0,0) circle(2);
\end{scope}
\begin{scope}
    \clip (-2,0) rectangle (0,-2);
    \filldraw[black,fill=black!20!white]  (-2,0) circle(2);
\end{scope}
\begin{scope}
    \clip (0,-2) rectangle (2,0);
    \filldraw[black,fill=black!20!white]  (2,0) circle(2);
\end{scope}
\end{tikzpicture}
\caption{A cuspidal type singularity for $\de E$ on $\de \Om$\label{fig:cuspidal}}
\end{subfigure}
\qquad
\begin{subfigure}[b]{.35\linewidth}
\centering
\begin{tikzpicture}[thick]
\begin{scope}[rotate=90]
%
\fill[black!20!white] (-2, -1) rectangle (-1, 1);
\fill[black!20!white] (1, -1) rectangle (2, 1);
\begin{scope}
    \clip (-1.1,-1) rectangle (1.5,1.5);
    \fill[black!20!white]  (-1,0) circle(1);
    \draw[dashed] (-1,0) circle(1);
\end{scope}
\begin{scope}
    \clip (1.1,-1.5) rectangle (0,1);
    \fill[black!20!white]  (1,0) circle(1);
    \draw[dashed]  (1,0) circle(1);
\end{scope}
\draw (-2,-1) -- (2,-1);
\draw (-2,-1) -- (2,-1);
\draw (-2,1) -- (2,1);
\end{scope}
\end{tikzpicture}
\caption{A rounded X-type singularity for $\de E \cap \Om$\label{fig:x_rounded}}
\end{subfigure}
}
\caption{The singularities occurring in \cref{lem:rollingball}. The boundary of $\Om$ is represented by the thick continuous lines.\label{fig:singularities}}
\end{figure}
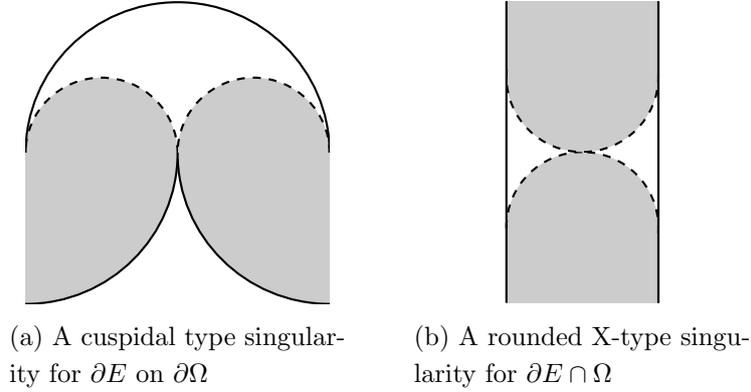
\end{proof}

\subsection{Additional properties when \texorpdfstring{$\Om$}{Om} is a Jordan domain}\label{ssec:jordan_domains}

When the set $\Om$ is a Jordan domain, satisfying the technical assumption $|\de\Om|=0$, the minimizers enjoy some additional properties, which we recall here. These were originally proved in~\cite{LNS17} for the case $\k=h_\Om$, and their proofs are easily adapted to a general $\k$, hence we shall omit the details here.

\begin{prop}\label{prop:palla_interna}
Suppose $\Om\subset \R^2$ is a Jordan domain with $|\de \Om| = 0$, and let $E_\k$ be a minimizer of $\cF_\k$. Then,
\begin{itemize}
\item[(i)] the curvature of $\de E_\k$ is bounded from above by $\k$ in both variational and viscous senses;
\item[(ii)] $E_\k$ is Lebesgue-equivalent to a finite union of simply connected open sets, hence its measure-theoretic boundary $\de E_\k$ is a finite union of pairwise disjoint Jordan curves;
\item[(iii)] $E_\k$ contains a ball of radius $\k^{-1}$.
\end{itemize}
\end{prop}
For the definitions of curvature in variational and in viscous senses we refer to, resp.,~\cite{BGM03} and~\cite[Definition~2.3]{LNS17}. The proof of~(i), for $\kappa\ge h_\Om$ is obtained as in~\cite[Lemma~2.2 and Lemma~2.4]{LNS17}. When $\kappa \in [R_\Om^{-1}, h_\Om)$ and $|E_\kappa|>\pi \kappa^{-2}$, one can analogously show that $E_\kappa$ is a $(\Lambda, r_0)$-minimizer of the perimeter (see \cite{Mag12book}) with $r_0 = r_0(|E_\kappa|)>0$, and the same reasoning applies. Whenever the minimizer is such that $|E_\kappa|=\pi \kappa^{-2}$, as we have already discussed immediately after \cref{prop:properties}, the minimizer needs to be a ball, and thus the claim is trivial. The proof of~(ii) follows from the same argument of~\cite[Propositions~2.9 and~2.10]{LNS17}. Point~(iii) follows from~(i) and~(ii) paired with~\cite[Theorem~1.6]{LNS17}.

\begin{rem}
We observe that \cref{prop:palla_interna} implies that every minimizer, for $\kappa\in [R_\Om^{-1}, h_\Om)$, has a unique P-connected component, that is the analog of connected component in the theory of sets of finite perimeter (see~\cite{ACMM01}). Indeed, the bound on the variational curvature stated in~(i) holds on every P-connected component. By~(ii) each of these components is a simply connected open set, whose boundary is a Jordan curve. Then,~(iii) holds for each component, and hence each component has volume at least $\pi \k^{-2}$ and thus it belongs to $\mathcal{C}(\k)$. Assume now by contradiction that $E_\kappa$ has more than one P-connected component, say without loss of generality $E^1_{\k}$ and $E^2_{\k}$, which, by the above discussion, are competitors for \eqref{eq:pmc_mod}. Recall that in the regime $\kappa\in [R_\Om^{-1}, h_\Om)$ one has $\min \mathcal{F}_\kappa>0$, and thus 
\[
\mathcal{F}_\kappa[E^i_{\k}] > 0\,, \qquad i=1,2\,.
\]
The contradiction now is reached, since $\mathcal{F}_\kappa[E_{\k}] = \mathcal{F}_\kappa[E^1_{\k}]+\mathcal{F}_\kappa[E^2_{\k}]$ and removing a component produces a competitor with a strictly smaller energy.
\end{rem}

\subsection{Characterization of the minimizers of \texorpdfstring{$\mathcal{F}_\kappa$}{Fk}}\label{ssec:proof_shape}

This subsection is devoted to the proof of \cref{thm:main_shape_min}. At the end, we will obtain as a corollary the monotonicity (in the set inclusion sense) of minimizers of $\mathcal{F}_\kappa$ with respect to $\k$.

First, we shall see that assuming that $\Omega$ has no necks of radius $r=\k^{-1}$, and combining \cref{prop:palla_interna}~(iii) and \cref{lem:rollingball} yield the following result.

\begin{cor}\label{cor:useful_cor}
Let $\Om$ be a Jordan domain with $|\de \Om|=0$ and let $\k > R_\Om^{-1}$ be fixed. Assume $\Om$ has no necks of radius $r=\k^{-1}$. Let $E_{\k}$ be a minimizer of ${\mathcal F}_{\k}$. Then,
\begin{itemize}
\item[(i)] $E_\k$ contains $C_0\oplus B_r$, where $C_0$ is defined as in \cref{prop:struttura_diff}, namely
\[
C_0 = \overline{\interior(\Om^{r})} \cup  \bigcup_{\gamma\in \Gamma^2_r} \gamma([0,1]) \,;
\]
\item[(ii)] there exists a unique maximal minimizer, $E^M_\k$;
\item[(iii)]$E^M_\k$ coincides with the union of all minimizers.
\end{itemize}
Points~(ii) and~(iii) hold the same for $\k=R_\Om^{-1}$.
\end{cor}

\begin{proof}
Let $\k > R_\Om^{-1}$ be fixed. By \cref{prop:palla_interna}~(iii) we know that $E_{\k}$ contains a ball $B_{r}(x)$ with $x\in \Om^{r}$. By the no neck assumption, it then must contain any ball of radius $r$ centered on $\overline{\interior(\Om^{r})}$. If this were not the case, we could find $y\in \interior(\Om^{r})$ such that $B_{r}(y)$ is not contained in $E_{\k}$, thus obtaining a contradiction with \cref{lem:rollingball} (i). Hence $E_{\k} \supset \overline{\interior(\Om^{r})} \oplus B_r$. By definition of $\Gamma^2_r$, and thanks to \cref{lem:rollingball}(ii), $E_{\k}$ must also contain any ball of radius $r$ centered on every point of $\g\in \Gamma^{2}_{r}$. This establishes point~(i).

Regarding point~(ii), argue by contradiction and assume there are two distinct maximal minimizers. Since they both contain $C_0\oplus B_r$, their intersection has volume at least $\pi \kappa^{-2}$. Hence by \cref{prop:closed_class}, their union is a viable competitor with greater volume. This is a contradiction and point~(ii) follows. Point~(iii) follows with the same reasoning, again exploiting point~(i).

For $\k=R_\Om^{-1}$, \cref{prop:palla_interna}~(iii) paired with \cref{lem:rollingball}~(iii) implies that $E_\kappa$ contains $\Om^{R_\Omega}\oplus B_{R_\Omega}$. This is enough to prove points~(ii) and~(iii), by following the same reasoning used for general $\k> R_\Om^{-1}$.
\end{proof}

Second, we recall a lemma about the so-called \emph{arc-ball property} for minimizers of $\mathcal{F}_\kappa$. This property was first shown for Cheeger sets in planar strips in~\cite{LP16} and later extended to Cheeger sets in Jordan domains without necks in~\cite{LNS17}. The proof follows that of~\cite[Theorem~1.4]{LNS17}.

\begin{lem}[Arc-ball property]\label{lem:arcball}
Let $\Om$ be a Jordan domain with $|\de\Omega|=0$, and let $\k\ge R_{\Om}^{-1}$ be fixed. Assume $\Om$ has no necks of radius $r=\k^{-1}$ and let $E_{\k}$ be a minimizer of $\cF_{\k}$. Then any connected component of $\de E_{\k}\cap \Om$ is contained in the boundary of a ball of radius $r$ contained in $E_{\k}$.
\end{lem}
\begin{proof}[Proof]
Let $\alpha$ be a connected component of $\de E_{\k}\cap \Om$. By \cref{prop:properties}~(i) we know that $\alpha$ is a circular arc of radius $r$ belonging to a ball $B_{r}(x)$. If $B_{r}(x)$ is contained in $E_{\k}$ we have nothing to prove. Otherwise, let $y\in \alpha$ be the midpoint of $\alpha$ and consider the largest $0<t<r$ such that setting $z_{t} = y + t\frac{(x-y)}{|x-y|}$ we have $B_{t}(z_{t})\subset E_{\k}$. One can now argue exactly as in the proof of~\cite[Theorem 1.4, pag.~21]{LNS17}.
\end{proof}

\begin{proof}[Proof of \cref{thm:main_shape_min}]
For $\k= h_\Om$ the characterization of minimal and maximal minimizers was proved in~\cite[Theorem~1.4]{LNS17}, and later extended to $\k\ge h_\Om$ in~\cite[Theorem~2.3]{LS20}. Here we generalize those previous proofs in order to deduce the complete classification of minimizers of the prescribed curvature problem. 

The proof of the structure of the maximal minimizer $E^{M}_{\k}$ in the case $\k \in [R_\Om^{-1}, h_\Om)$ is essentially the same as the one for $\k \ge h_\Om$. By \cref{cor:useful_cor}~(ii), we already know that it is unique. The assumption of no necks of radius $r$ paired with \cref{lem:rollingball}(iii) implies that $E^M_\k \supseteq \Om^r \oplus B_r$. The opposite inclusion follows by reasoning exactly as in the proof of Theorem~1.4 in~\cite{LNS17}.

\textit{Proof of (i)}. We assume $r = \k^{-1} <R_{\Om}$ and $\Gamma^{1}_{r}\neq \emptyset$. We start by proving that $C_{\theta}\oplus B_{r}$ minimizes ${\mathcal F}_{\k}$ for any $\theta:\Gamma_{r}^{1}\to [0,1]$, using the fact that $E_{\k}^{M} = C_{1}\oplus B_{r} = \Om^{r}\oplus B_{r}$. Indeed, by \cref{prop:struttura_diff} we know that $C_{\theta}$ is contractible and satisfies $\reach(C_{\theta})\ge r$. Therefore we can use Steiner's formulas, see for instance~\cite{FedererCM} and~\cite[Section~2.3]{LNS17}, and write
\begin{align}
|C_{\theta}\oplus B_{r}| = |C_{\theta}| + r\mathcal{M}_o(C_{\theta}) + \pi r^2\,, && P(C_{\theta}\oplus B_{r}) = \mathcal{M}_o(C_{\theta}) + 2\pi r\,,\label{eq:steiner_C_theta}
\end{align}
where $\mathcal{M}_o(F)$ is the \emph{outer Minkowski content of $F$}, i.e.,
\[
\mathcal{M}_o(F) = \lim_{t\to 0} \frac{|F\oplus B_t|-|F|}{t}\,.
\]
Since $|C_{\theta}|=|C_1|$ by definition, using~\eqref{eq:steiner_C_theta} and $r=\kappa^{-1}$, it is immediate to check that the equality $\cF_\k[C_{\theta}\oplus B_{r}]=\cF_\k[E^M_\k]$ holds. By \cref{cor:useful_cor}~(i), this implies in particular that $C_{0}\oplus B_{r}$ is the unique minimal minimizer.\par

Let now $E_\k$ be any minimizer. By \cref{cor:useful_cor}~(i), it contains $C_{0}\oplus B_{r}$. Given $\theta,\theta':\Gamma_{r}^{1}\to [0,1]$, we write $\theta \le \theta'$ if $\theta(\g) \le \theta'(\g)$ for every $\g\in \Gamma_{r}^{1}$. Let $\theta_{r}$ be maximal (in the sense of the order relation $\le$) among those $\theta$ for which $C_{\theta}\oplus B_{r}\subset E_{\k}$. 

In order to conclude that $E_{\k} = C_{\theta_{r}}\oplus B_{r}$, we only need to show the inclusion $E_{\k}\subset C_{\theta_{r}}\oplus B_{r}$. We can assume without loss of generality that $\theta_{r} \not\equiv 1$, as otherwise $E_{\k}$ would coincide with the maximal minimizer $E_{\k}^M$ given by $C_{1}\oplus B_{r}$. By \cref{cor:useful_cor}~(iii) $E_\k \subset E_\k^M$, therefore we can find $\g\in \G_{r}^{1}$ and a point $z$ in the interior of $E_{\k} \setminus (C_{\theta_{r}}\oplus B_{r})$, such that $z\in B_{r}(\g(\tau))$ for some $\tau \in (\theta_{r}(\g),1)$. 

We now have the following alternative: either it is $B_{r}(\g(\tau)) \subset E_{\k}$, or it is $\de E_{\k}\cap B_{r}(\g(\tau)) \neq \emptyset$. 

In the first case, using \cref{lem:rollingball}~(ii), and the no neck assumption, immediately gives a contradiction to the maximality of $\theta_r$. In the second case the intersection $\de E_{\k}\cap B_{r}(\g(\tau))$ must necessarily be a single arc of curvature $\k$ contained in a connected component $\alpha$ of $\de E_{\k}\cap \Om$. By \cref{lem:arcball} there exists a ball $B_{r}(y)\subset E_{\k}$ such that $\alpha \subset \de B_{r}(y)$, and with $z\in B_{r}(y)$ and $y\notin C_{\theta_r}$. Again, by using \cref{lem:rollingball}(ii) we reach a contradiction with the maximality of $\theta_{r}$.
\medskip

\textit{Proof of (ii)}. This is immediate because, as in the previous case, we have $\Om^{r}\oplus B_{r} \subset E_{\k}$. However, thanks to the assumption $\G_{r}^{1}=\emptyset$, we also have $E_{\k}\subset E_{\k}^{M} = \Om^{r}\oplus B_{r}$, which gives~(ii) at once.
\medskip

\textit{Proof of (iii)}. By \cref{prop:struttura_diff} (a) we have that either $\Om^{r}$ is reduced to a point, or it is a $C^{1,1}$-diffeomorphic image of the interval $[0,1]$. In the first case, the inball of $\Om$ is easily seen to represent the unique minimizer. In the second case if $[a,b]\subset [0,1]$ and $\g:[0,1]\to \Om^{r}$ is a $C^{1,1}$-diffeomorphism with curvature bounded by $\k$, one can show that $\g([a,b])\oplus B_{r}$ is a minimizer, arguing as we did for~(i). Viceversa, let $E_{\k}$ be a minimizer, for which we know that there exists a ball $B_{r}(x)$ centered at some $x\in \Om^{r}$ and contained in $E_{\k}$. Denote by $[a,b]$ the largest closed subinterval of $[0,1]$ such that $K:= \g([a,b])$ contains $x$ and $K\oplus B_{r}\subset E_{\k}$. Of course, if $[a,b]=[0,1]$ we conclude that $E_{\k}= E^{M}_{\k} = \Om^{r}\oplus B_{r}$. Otherwise, we argue as in the proof of (i) and finally obtain the opposite inclusion $E_{\k} \subset K\oplus B_{r}$. This finally shows (iii) and concludes the proof of the theorem.
\end{proof}

\begin{cor}\label{cor:nestedness}
Let $\Om$ be a Jordan domain such that  $|\de \Om|=0$ and let $\k_2>\k_1\ge R_\Om^{-1}$. If $\Om$ has no necks of radius $\k_2^{-1}$ and $\k_1^{-1}$, then one has
\[
E^M_{\k_2} \supseteq E^m_{\k_2}  \supseteq E^M_{\k_1} \supseteq E^m_{\k_1}.
\]
\end{cor}

\begin{proof}
One reasons in the same way as in~\cite[Corollary~5.6]{LS20}. Fix $r_i = \k_i^{-1}$ for $i=1,2$, with $r_2<r_1$. Trivially the set $\overline{\interior(\Om^{r_2})}$ contains $\Om^{r_1}$, and thus $\overline{\interior(\Om^{r_2})}\oplus B_{r_2}$ contains $\Om^{r_1}\oplus B_{r_1}$. The claim immediately follows from \cref{thm:main_shape_min}.
\end{proof}

\begin{rem}\label{rem:nestedness}
The inclusion $E^m_{\k_2}  \supseteq E^M_{\k_1}$ is strict as soon as $\k_{1}<\k_{2}$ and $0<|E^M_{\k_1}|<|\Om|$. Indeed, this information on the volume is equivalent to say $E^M_{\k_1} \neq \Om$ and $E^M_{\k_1} \neq \emptyset$. Assuming then $E^M_{\k_1} = E^m_{\k_2}$ one obtains $\de E^M_{\k_1} \cap \Om=\de E^m_{\k_2} \cap \Om\neq \emptyset$, whence $\k_1=\k_2$, a contradiction.
\end{rem}

\section{The isoperimetric profile}\label{sec:isop_prof}
 
In this section we use \cref{thm:main_shape_min} to characterize all isoperimetric sets of a Jordan domain $\Omega$ with no necks of any radius. The full characterization is the content of \cref{thm:m_iso}, and this will be employed to prove some convexity properties of the isoperimetric profile, see \cref{sec:convexity}. Before the proof of \cref{thm:m_iso}, that is the core of the section, we need to prove the following lemma. 
 
\begin{lem}\label{lem:upsc}
Let $R>0$ be fixed and let $\Om$ be a Jordan domain with no necks of radius $r$, for all $r\le R$. Let $m(r)$ and $\mu(r)$ be the functions defined as
\begin{align*}
m(r) = \mathcal{M}_o(\Om^r)\,, && \mu(r) = \mathcal{M}_o(\overline{\interior(\Om^r)})\,.
\end{align*}
Then, $m$ is upper semicontinuous in $(0,R]$, while $\mu$ is lower semicontinuous in $(0,R)$. Consequently, one has
\begin{equation}\label{eq:usclsc}
\limsup_{r\to \hat{r}} |E^{M}_{r^{-1}}| \le |E^{M}_{\hat{r}^{-1}}|,\quad \liminf_{r\to \hat{r}} |E^{m}_{r^{-1}}| \ge |E^{m}_{\hat{r}^{-1}}|
\end{equation}
for every $\hat{r}\in (0,R]$.
\end{lem}

\begin{proof}
Thanks to~\cite[Lemma~6.1]{LS20}, we already have the upper (resp., lower) semicontinuity of $m$ (resp., $\mu$) on the open interval $(0,R)$. The very same proof yields as well the upper semicontinuous up to $R$ included of $m$, and thus we refer the reader to the original one. Then, by coupling Steiner's formulas
\begin{align*}
|E^M_{r^{-1}}|&=\pi {r}^2 + {r}\,m(r) + |\Om^{r}|\,,\\ 
|E^m_{r^{-1}}|&=\pi {r}^2 + {r}\,\mu(r) + |\overline{\interior(\Om^{r})}|
\end{align*}
with the properties of $m(r)$ and $\mu(r)$, and the fact that $|\Om^{r}| = |\interior(\Om^{r})|$ for each $r>0$, we obtain \eqref{eq:usclsc}.
\end{proof}

\begin{proof}[Proof of \cref{thm:m_iso}]
Assume $\pi R_{\Om}^{2} < V < |\Om|$ without loss of generality, and note that the thesis is a consequence of the following claim: 
\begin{equation}\label{eq:claim}
\exists\,\hat{\k}\ge R_{\Om}^{-1},\ \exists\,E_{\hat{\k}}\in \argmin{\mathcal F}_{\hat{\k}}\ \text{ with } |E_{\hat{\k}}|=V\,. 
\end{equation}
Indeed, let $E_{V}$ be such that $|E_{V}|=V$ and $P(E_{V})={\mathcal J}(V)$. Assuming that \eqref{eq:claim} is verified, we would deduce that $E_{V}$ minimizes ${\mathcal F}_{\hat{\k}}$ since 
\[
P(E_{V}) - \hat{\k} V \le P(E_{\hat{\k}}) - \hat{\k} V = \inf_{C(\hat{\k})}{\mathcal F}_{\hat{\k}}
\]
(note that the previous inequality is in fact an identity). 
In order to prove \eqref{eq:claim}, we define
\[
\k^*=\inf\{\,\k: |E^M_\k|>V\,\}, \qquad \qquad \k_*=\sup\{\,\k: |E^m_\k|<V\,\}\,.
\]
Notice that both sets are nonempty, hence the infimum and supremum are finite. Indeed, being $\Om$ an open set, we can approximate it in $L^1$ by sets of the form $\Om^r \oplus B_r$, for $0<r\le R_\Om$, which shows that the first set is nonempty. The second set is obviously nonempty, as it contains at least the curvature of the inball of $\Om$ by \cref{thm:main_shape_min}~(iii).

We now claim that $\k_{*} = \k^{*}$. Indeed, let us first assume by contradiction that $\k^*< \k_*$. Then we would find two curvatures $\kappa_1, \kappa_2$ such that
\[
\kappa^* < \kappa_1 < \kappa_2 < \kappa_*\,.
\]
By definition of $\kappa^*$ and of $\kappa_*$ we would have
\[
|E^m_{\kappa_2}| < V < |E^M_{\kappa_1}|\,,
\]
against the fact $E^m_{\kappa_2} \supseteq E^M_{\kappa_1}$, as granted by \cref{cor:nestedness}. By a similar argument we can also exclude the case $\k_{*}<\k^{*}$. Indeed, assume that the strict inequality holds and, for any $\k \in (\k_*, \k^*)$, let $E_\k$ be a minimizer of $\mathcal{F}_\k$.  On the one hand, the lower bound $\k > \k_*$ implies $|E^m_\k| \ge V$, while the upper bound $\k < \k^*$ implies $|E^M_\k| \le V$. Hence, we have $|E^M_\k|=|E^m_\k|=V$ for all $\k \in (\k_*, \k^*)$. Since $V< |\Omega|$, this yields a contradiction with \cref{cor:nestedness} (see also \cref{rem:nestedness}).\par

Let now $\hat \k := \k^*=\k_*$ and $\hat r=1/\hat \k$, hence by \cref{lem:upsc} we infer that
\[
|E^M_{\hat\k}| \ge V\ge |E^m_{\hat \k}|\,.
\]
Now, if one of these two inequalities is an equality, we are done. If they are both strict, by \cref{thm:main_shape_min} we necessarily have that  $\Gamma^1_{\hat r}$ is not empty. As the parametrized sets $C_{\theta}\oplus B_{r}$ and $K\oplus B_{r}$, defined in \cref{thm:main_shape_min}, form a family of minimizers with volumes varying continuously from $|E^m_{\hat \k}|$ up to  $|E^M_{\hat \k}|$, we can always find one of them with volume exactly $V$, which proves~\eqref{eq:claim} and thus the theorem.
\end{proof}

\begin{rem}
Nestedness of the isoperimetric sets is not true in general. However, with reference to \cref{thm:main_shape_min}, one can always select a one-parameter family of nested isoperimetric sets. In the case $r<R_{\Om}$, this family is of the form $C_{\theta_{s}}\oplus B_{r}$. The choice of $\theta_{s}$ can be made in such a way that the map $s\mapsto |C_{\theta_{s}}\oplus B_{r}|$ is continuous, $\theta_{s}$ is increasing in $s\in [0,1]$ with respect to the order relation $\le$, $\theta_{0}\equiv 0$, and $\theta_{1} \equiv 1$. Similarly, in the case $r=R_{\Om}$, one can set $K_{s}$ as the image of the interval $[0,s]$ through the $C^{1,1}$-parametrization $\g$ of the set $\Om^{r}$, and consider the nested family $K_{s}\oplus B_{r}$. 
\end{rem}

\begin{cor}\label{cor:mapisop}
Let $\Om$ be a Jordan domain with $|\de \Om|=0$ and without necks of radius $r$, for all $r\in(0, R_\Om]$. Suppose that $\Om^{R_\Om}=\{x\}$. Then the following hold:
\begin{itemize}
\item[(i)] if the set $\G^1_r$ consists of at most one curve for all $r<R_\Om$, then there exists a unique isoperimetric set for all $V\in [\pi R_{\Om}^{2},|\Omega|)$;
\item[(ii)] if the set $\G^1_r$ is empty for all $r<R_\Om$, and if we let $\k$ be the curvature of $\de E\cap \Om$, where $E$ is the unique isoperimetric set of volume $V\ge \pi R^2_\Om$, then the map $\Phi(V) = \k$ is a bijection.
\end{itemize} 
\end{cor}

\begin{proof}
Let $V \ge  \pi R^2_\Om$ be fixed, and let $\kappa$ be the unique curvature such that $|E^m_\kappa|\le V\le |E^M_\kappa|$, as in the proof of \cref{thm:m_iso}. If $E^m_\kappa$ and $E^M_\kappa$ coincide, i.e., $\G^1_r=\emptyset$, uniqueness follows. If otherwise $\G^1_r$ consists of exactly one curve $\gamma$, then by \cref{thm:main_shape_min} all the possible minimizers of $\mathcal{F}_\kappa$ are given by the one-parameter family
\[
E_t:=\left(\overline{\interior(\Om^r)} \cup \gamma(0,t)\right) \oplus B_r\,.
\]
As $t\mapsto |E_t|$ is strictly increasing, there exists a unique $t\in [0,1]$ such that the corresponding minimizer has volume exactly $V$, hence the uniqueness.

Regarding the second point, by \cref{thm:main_shape_min}, we already know that for each volume $V$ there exists a unique curvature $\kappa$, thus we are left to prove that the map is injective. Under the assumptions of the corollary, for each $\kappa$ we have the set equality $E^m_\kappa=E^M_\kappa$. Pairing this with the nestedness property granted by \cref{cor:nestedness}, yields the injectivity.
\end{proof}

\begin{cor}
Let $\Om$ be a Jordan domain with $|\de \Om|=0$ and without necks of radius $r$, for all $r\in(0, R_\Om]$. For all $V\ge \pi R^2_\Om$, let $\Phi(V)$ be the map defined as in \cref{cor:mapisop}~(ii). Then the image $\Phi([\pi R^{2}_{\Om}, |\Om|))$ is a finite interval $[R_\Om^{-1}, \bar \k)$ if and only if $\Om$ satisfies an interior ball condition of curvature $\bar \k$, that is, $\reach(\R^2 \setminus \Om)\ge \bar \kappa^{-1}$.
\end{cor}

\begin{proof}
First notice that in general one has
\[
\Om = \lim_{r \to 0}\big(\Om^{r} \oplus B_{r}\big)\,.
\]
By the characterization of minimizers given in \cref{thm:main_shape_min}, and the nestedness granted by \cref{cor:nestedness}, the image $\Phi([\pi R^{2}_{\Om}, |\Om|))$ is a finite interval $[R_\Om^{-1}, \bar \k)$ if and only if
\[
\Om = \lim_{\k \to \bar \k}\big(\Om^{r_\k} \oplus B_{r_\k}\big) = \Om^{r_{\bar \k}} \oplus B_{r_{\bar \k}}\,,
\]
where $r_\k = \k^{-1}$ as usual, and the limit is meant in a set-wise sense. This happens if and only if $\Om$ minimizes $\mathcal{F}_{\bar \k}$. If $\Om$ is such a minimizer, then the claim is immediate. Conversely, an interior ball condition of radius $\bar r=\bar \kappa^{-1}$ implies $\Om = \Om^{\bar r}\oplus B_{\bar r}$, see~\cite[Lemma~3.1]{Sar21}, thus the opposite claim is as well established.
\end{proof}

\subsection{Convexity properties}\label{sec:convexity}

In this section we establish some convexity properties of the isoperimetric profile $\mathcal{J}$ and of its square $\mathcal{J}^2$. The key point is to prove that, for $V\in [\pi R_{\Omega}^{2},|\Omega|]$, $\mathcal{J}(V)$ is the Legendre transform of a convex function.

We remark that the following is the natural extension of Proposition~6.2 in~\cite{LS20}, which allowed us to establish the Legendre duality for the smaller interval $[|E_{h_\Omega}^m|, |\Omega|]$. Exploiting \cref{thm:m_iso}, we can prove that this duality holds true for the larger interval $[\pi R_{\Omega}^{2},|\Omega|]$.

\begin{prop}\label{prop:legendre_transform}
Let $\Om$ be a Jordan domain with $|\de \Om|=0$ and with no necks of radius $r=\kappa^{-1}$ for all $r\in(0, R_\Om]$. Then, the isoperimetric profile $\mathcal{J}(V)$ restricted to $V\in [\pi R^2_\Om,|\Om|)$ is the Legendre transform of 
\[
\mathcal{G}(\k) = -\min_{E\in \mathcal{C}(\k)} \mathcal{F}_\k(E)
\]
restricted to $\k \ge R_\Om^{-1}$, where $\mathcal{C}(\k)$ is defined as in~\eqref{eq:classCk}.
\end{prop}

\begin{proof}
The convexity of $\mathcal{G}$ follows as in the proof of~\cite[Proposition~6.2]{LS20}. Indeed, thanks to the geometric characterization of the minimizers obtained in \cref{thm:main_shape_min}, and for any admissible choice $\kappa \ge R_\Om^{-1}$, all minimizers of $\mathcal{F}_\kappa$ contain at least a ball of radius $R_\Om$. Therefore, for any choice of $\kappa$, one can minimize over the smaller class of competitors $\mathcal{C}(R_\Om^{-1})$ in place of the larger one $\mathcal{C}(\kappa)$ without affecting the value of the minimum or the minimizers.\par
We are left to show that the Legendre transform $\mathcal{G}^*$ of $\mathcal{G}$ coincides with the isoperimetric profile $\mathcal{J}$ on the claimed interval. By definition the Legendre transform is
\begin{align}
\mathcal{G}^*(V) &:= \sup_{\k \ge R_\Om^{-1}} \{\,\k V -\mathcal{G}(\k)\,\} \label{def:legendre_transf}\\
&= \sup_{\k \ge R_\Om^{-1}} \big\{\,\k V +\min_{E\in C(R_\Om^{-1})}\{P(E)-\k|E|\,\} \big\}\,,\nonumber
\end{align}
and we refer the interested reader to~\cite[Part~V, Chap.~26]{Rock15book} for the basic definitions and results on the Legendre transform. By \cref{thm:m_iso} for all $V\ge |B_R|$ there exist (a unique) $\bar\k\ge R_\Om^{-1}$ and a minimizer $E_{\bar \k}$ of $\cF_{\bar \k}$ with $|E_{\bar \k}|=V$ and such that $\mathcal{J} (V)= P(E_{\bar \k})$. Hence, on the one hand
\begin{align*}
\mathcal{G}^*(V) &\ge \bar \k V + \min_{E\in C(R_\Om^{-1})}\{\,P(E)-\bar \k|E|\,\} \\
&= \bar \k V + P(E_{\bar \k}) - \bar \k |E_{\bar \k}|=  P(E_{\bar \k}) = \mathcal{J} (V).
\end{align*}
On the other hand, for all $\k$ we have
\begin{align*}
\k V -\mathcal{G}(\k) &= \k V +\min_{E\in C(R_\Om^{-1})} \cF_\k(E) \\
&\le \k V + P(E_{\bar \k}) -\k |E_{\bar \k}| = P(E_{\bar \k}).
\end{align*}
Thus, by this inequality and~\eqref{def:legendre_transf} one has
\begin{align*}
\mathcal{G}^*(V) &\le P(E_{\bar \k}) = \mathcal{J} (V),
\end{align*}
and the claim follows at once.
\end{proof}

\begin{rem}
We remark that the isoperimetric profile $\mathcal{J}$ is differentiable at $V\in (0,|\Om|)$, and its derivative is given by the curvature of $\de E\cap \Om$, being $E$ any isoperimetric set of volume $V$. For volumes less than or equal to $\pi R^2_\Om$ it is an immediate computation. For volumes above this threshold it follows from the fact that $\mathcal{J}$ coincides with the Legendre transform of $\mathcal{G}$. To see this, we first note that $\mathcal{G} = (\mathcal{G}^*)^*$, since $\mathcal{G}$ is convex and lower semicontinuous. Second, from the equalities $\mathcal{G} = (\mathcal{G}^*)^*$ and $\mathcal{G}^* = \mathcal{J}$  we have
\begin{equation}\label{eq:doppia_trasformata}
\mathcal{G}(\overline{\kappa}) = \sup_{V \ge \pi R_\Om^2} \{\overline{\kappa} V - \mathcal{J}(V)\} = \overline{\kappa} \overline V - \mathcal{J}(\overline V)\,,
\end{equation}
for some (possibly non unique) $\overline{V}$. The equalities in~\eqref{eq:doppia_trasformata} imply that $\overline{\kappa}$ belongs to the subdifferential $\de \mathcal{J}(\overline{V})$ because
\[
\overline{\kappa} \overline{V} - \mathcal{J}(\overline{V}) \ge \overline{\kappa}V' - \mathcal{J}(V')
\]
for all $V'\ge \pi R_{\Omega}^{2}$. Now one concludes, since \cref{thm:m_iso} implies that for any volume $\overline{V}$ there exists a unique curvature $\kappa$, (and necessarily $\kappa = \overline{\kappa}$), for which~\eqref{eq:doppia_trasformata} is attained by $\overline{V}$. Thus the subdifferential $\de \mathcal{J}(\overline{V})$ reduces to the single element $\overline \kappa$, and being $\mathcal{J}$ convex this means that it is differentiable at $\overline{V}$ with derivative given by $\overline \kappa$. We also note that the link between the derivative of the isoperimetric profile and the (mean) curvature of the (internal) boundary of the minimizer is a classical fact, see for instance~\cite{Ros05}.
\end{rem}

\begin{rem}\label{rem:hp_n_dim}
By closely following the proof of \cref{prop:legendre_transform}, one notices that the key ingredient is to find a minimizer of $\mathcal{F}_\k$ for every choice of volume above a certain threshold $\overline{V}$, in this case $\overline{V}\ge \pi R^2_\Om$.
\end{rem}

\begin{rem}
Thanks to our geometric characterization, one can still obtain convexity of $\mathcal{J}$ on subintervals of $[\pi R^2_\Om, |\Om|]$ for sets $\Om$ which have no necks of radius $r$ with $r\in [r_1, r_2]$. Indeed, \cref{thm:m_iso} can be applied for all volumes $V\in [|E^m_{r_2^{-1}}|, |E^M_{r_1^{-1}}|]$, finding then a suitable curvature $\bar \k \in [r_2^{-1},r_1^{-1}]$. Thus, one obtains convexity of $\mathcal{J}$ on such an interval of volumes by following the proof of \cref{prop:legendre_transform}.
\end{rem}

\begin{rem}
It is easy to verify that whenever $\Gamma^1_r\neq \emptyset$, the isoperimetric profile $\mathcal{J}$ is linear in the interval of volumes $[|E^m_{r^{-1}}|, |E^M_{r^{-1}}|]$, and viceversa. There exist sets with no necks of radius $r$ for all $r\le R_\Om$ that have such a linear growth on countably many intervals (of volume), i.e., such that $\Gamma^1_r\neq \emptyset$ for countably many $r$, see for instance~\cite[Example~5.8 and Figure~4]{LS20}.
\end{rem}

We are now ready to prove \cref{thm:convexity_J} which establishes the convexity properties of $\mathcal{J}$ and of $\mathcal{J}^2$ .

\begin{proof}[Proof of \cref{thm:convexity_J}]
On the one hand, as the Legendre transform maps convex maps into convex maps, one immeditaly obtains by \cref{prop:legendre_transform} the convexity of $\mathcal{J}$ for $V\ge \pi R^2_\Om$. Therefore $\mathcal{J}^2$ is convex as well on such interval. On the other hand, for volumes $V$ up to $\pi R^2_\Om$, any ball of volume $V$ is a minimizer. An immediate computation gives
\begin{align*}
\mathcal{J}^2(V)= 4\pi V\,,\qquad \text{for } V\in [0, \pi R^2_\Om]\,,
\end{align*}
which is linear and thus convex. We are then left to show that the piecewise convex function $\mathcal{J}^2$ is globally convex. First, notice that $\mathcal{J}$ is continuous and so it is $\mathcal{J}^2$. Second, recall that a function $f:[a,b]\to \mathbb{R}$ is convex if and only $f'_-(x)\le f'_+(x)$ for all $x\in (a,b)$. To conclude the claim, it is enough to show that $(\mathcal{J}^{2})'_{-}(V_{0}) \le (\mathcal{J}^{2})'_{+}(V_{0})$ at $V_{0}=\pi R^2_\Om$. Clearly both the left and right derivatives are well defined, and one trivially has
\[
(\mathcal{J}^2)'_{-} (V_0)= 4\pi\,.
\]
Let us now denote by $\mathcal{I}(V)$ the isoperimetric profile of $\mathbb{R}^2$, i.e.,
\[
\mathcal{I}(V):=\inf\{\, P(E)\,:\, |E|=V\,, E\subset \mathbb{R}^2 \,\}\,,
\]
for which we know $\mathcal{I}^2(V) = 4\pi V$. By the isoperimetric inequality we have
\begin{align*}
\mathcal{I}^2(V) \le \mathcal{J}^2(V)\,,\qquad \forall\, V>0\,.
\end{align*}
The previous inequality, paired with $\mathcal{I}^2(V_{0}) = \mathcal{J}^2(V_{0})$, immediately implies
\[
(\mathcal{J}^2)'_{-}(V_{0}) = (\mathcal{I}^2)'(V_{0}) \le  (\mathcal{J}^2)'_{+}(V_{0})\,, 
\]
whence the claim follows.
\end{proof}

\subsection{Few words on dimension \texorpdfstring{$n$}{n}}\label{sec:dim_n}

As seen in the previous section, \cref{thm:m_iso} provides solutions of the isoperimetric problem by solving a (partially) unconstrained problem, i.e., the minimization of $\mathcal{F}_\kappa$ on the class $\mathcal{C}(\k)$. In the planar setting of Jordan domains $\Om$ satisfying a no neck property, a full geometric characterization of minimizers is provided. Unfortunately, in higher dimension $n\ge 3$, we cannot expect to find such a precise characterization of minimizers of $\mathcal{F}_\kappa$ for a given, bounded, open set $\Om\subset \R^{n}$. Nevertheless, as noticed in \cref{rem:hp_n_dim}, the convexity of the profile above some threshold $\overline{V}$ would simply follow by knowing that among the minimizers of $\mathcal{J}$ for volumes greater than $\overline{V}$, one can find minimizers of $\mathcal{F}_\kappa$ for suitable $\kappa$.

This program has been partially carried on in~\cite{ACC05}, for convex sets $\Om\subset \mathbb{R}^n$ of class $C^{1,1}$ and with $V\ge |E_{h_\Om}|$, being $E_{h_\Om}$ the unique Cheeger set of $\Om$, as we hereafter recall.

\begin{thm}[Theorem~11 of~\cite{ACC05}]\label{thm:ACC05}
Let $\Om\subset \mathbb{R}^n$ be convex and of class $C^{1,1}$. Then, for each volume $|E_{h_\Om}|\le V \le |\Om|$ there exists $\kappa\in [h_{\Om}, +\infty)$ such that the unique minimizer $E_\kappa$ of $\mathcal{F}_\kappa$ satisfies
\begin{align*}
|E_\kappa|= V\,, && \mathcal{J}(V) = P(E_\kappa)\,.
\end{align*} 
\end{thm}

Thanks to \cref{thm:ACC05} we immediately obtain the following convexity property of the isoperimetric profile $\mathcal{J}$ for $\Om\subset \R^{n}$ convex, bounded and of class $C^{1,1}$, in any dimension $n\ge 2$.

\begin{thm}
Let $n\ge 2$ and let $\Om\subset \mathbb{R}^n$ be convex, bounded, and of class $C^{1,1}$. Then the isoperimetric profile $\mathcal{J}$ is convex in $[|E_{h_\Om}|, |\Om|]$.
\end{thm}

\begin{proof}
Immediate consequence of \cref{rem:hp_n_dim} and \cref{thm:ACC05}.
\end{proof}

The proofs in~\cite{ACC05} which lead to \cref{thm:ACC05} heavily rely on the comparison principle and Korevaar's concavity maximum principle (see~\cite{Kor83a, Kor83b}) to study $\mathcal{F}_\kappa$, with $\kappa\ge h_\Om$. It would be of great interest to see if the approach of~\cite{ACC05} can be extended to curvatures $\kappa\in [R_{\Om}^{-1},h_\Om)$ by considering the same functional $\mathcal{F}_\kappa$ with the additional lower bound to the volume, just as we did here, in order to prevent the empty set to be a minimizer. Were it possible to extend their approach to $\kappa \in [R_\Om^{-1}, h_\Om)$, one would easily find out that $\mathcal{J}^\frac{n}{n-1}$ is convex, just by repeating the arguments of our \cref{prop:legendre_transform} and \cref{thm:convexity_J}.

\section{Examples}\label{sec:examples}

In this final section we explicitly compute and plot the isoperimetric profile for two class of sets: rectangles, and cross-shaped sets, which are non convex.

\subsection{Rectangles}

Take a rectangle which, up to scaling, has a fixed side of length $2$ and the other side of length $L\ge 2$, and denote it as $\mathcal{R}_L$. The inradius of $\mathcal{R}_L$ is $1$. Notice that for all $r< 1$ the set $\Gamma^1_r$ is empty, therefore for any prescribed curvature $\kappa> 1$ there exists a unique minimizer of $\mathcal{F}_\k$. We have the following
\begin{itemize}
\item[i)] for volumes $V\le \pi$, the minimizer is a ball of radius $\sqrt{\pi^{-1}V}$;
\item[ii)] for volumes $V\in (\pi, \pi+  2(L-2)]$, any minimizer is the convex hull of two balls of radius $1$;
\item[iii)] for volumes $V>  \pi +  2(L-2)$, there exists $\kappa > 1$ such that $E_\kappa = (\mathcal{R}_L)_{\kappa^{-1}} \oplus B_{\kappa^{-1}}$ is the unique minimizer.
\end{itemize}

An easy computation yields
\[
\mathcal{J}(V) =
\begin{cases}
2\sqrt{\pi V}\,, & V\in [0, \pi]\\
\pi + V\,, &V\in (\pi, \pi + 2(L-2)]\\
-2\sqrt{4-\pi}\sqrt{2L-V}+2(2+L)\,, &V\in (\pi + 2(L-2), 2L]
\end{cases}
\]
By \cref{thm:convexity_J} we already know that $\mathcal{J}^2$ is globally convex. As a visual confirmation of this fact, in \cref{fig:plot_J_R} (resp., \cref{fig:plot_J^2_R}) one can see the plot of $\mathcal{J}$ (resp.,~$\mathcal{J}^2$) for $\mathcal{R}_4$.

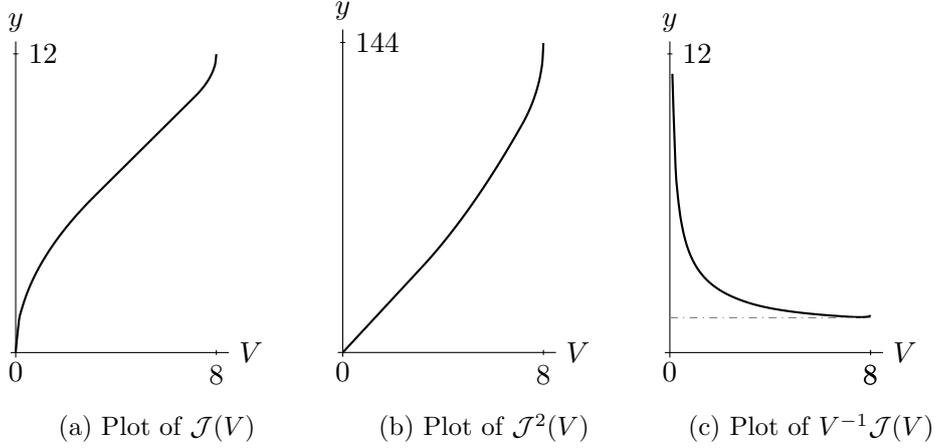
\begin{figure}[t]
\makebox[\linewidth][c]{
\begin{subfigure}[b]{.33\linewidth}
\begin{tikzpicture}[x=.33cm, y=.33cm]
\draw[-] (-0.2, 0) -- (8.5, 0) node[right] {$V$};
\draw[-] (0,-0.2) -- (0, 12.5) node[above] {$y$};
\draw[domain=0:pi, smooth, variable=\x, thick] plot ({\x}, {(2*pi^(0.5)*\x^(0.5))});
\draw[domain=pi:pi+4, smooth, variable=\x, thick] plot ({\x}, {pi+\x});
\draw[domain=pi+4:8, smooth, variable=\x, thick] plot ({\x}, {12-2*(4-pi)^(0.5)*(8-\x)^(0.5)});
\draw (-0.1,12) -- (0.1,12) node[right] {$12$};
\draw (8,0.1) -- (8,-0.1) node[below] {$8$};
\node[below] at (0,0) {$0$};
\end{tikzpicture}
\caption{Plot of $\mathcal{J}(V)$\label{fig:plot_J_R}}
\end{subfigure}
\begin{subfigure}[b]{.33\linewidth}
\begin{tikzpicture}[x=.33cm, y=.33cm, yscale=0.0865]
\draw[-] (-0.2, 0) -- (8.5, 0) node[right] {$V$};
\draw[-] (0,-2) -- (0, 150) node[above] {$y$};
\draw[domain=0:pi, smooth, variable=\x, thick] plot ({\x}, {4*pi*\x});
\draw[domain=pi:pi+4, smooth, variable=\x, thick] plot ({\x}, {(pi+\x)^2});
\draw[domain=pi+4:8, smooth, variable=\x, thick] plot ({\x}, {(12-2*(4-pi)^(0.5)*(8-\x)^(0.5))^2});
\draw (-0.1,144) -- (0.1,144) node[right] {$144$};
\draw (8,1.2) -- (8,-1.2) node[below] {$8$};
\node[below] at (0,0) {$0$};
\end{tikzpicture}
\caption{Plot of $\mathcal{J}^2(V)$\label{fig:plot_J^2_R}}
\end{subfigure}
\begin{subfigure}[b]{.33\linewidth}
\begin{tikzpicture}[x=.33cm, y=.33cm, spy using outlines=
	{circle, magnification=4, connect spies}]
\draw[-] (-0.2, 0) -- (8.5, 0) node[right] {$V$};
\draw[-] (0,-0.2) -- (0, 12.5) node[above] {$y$};
\draw[domain=0.1:pi, smooth, variable=\x, thick] plot ({\x}, {(2*pi^(0.5)/\x^(0.5))});
\draw[domain=pi:pi+4, smooth, variable=\x, thick] plot ({\x}, {(pi+\x)/\x});
\draw[domain=pi+4:8, smooth, variable=\x, thick] plot ({\x}, {(12-2*(4-pi)^(0.5)*(8-\x)^(0.5))/\x});
\draw[domain=0:8, smooth, variable=\x, dash dot, thin, gray] plot ({\x}, {(4-pi)/(2*(3-(1+2*pi)^(0.5)))-.025});
\draw (-0.1,12) -- (0.1,12) node[right] {$12$};
\draw (8,0.1) -- (8,-0.1) node[below] {$8$};
\draw (8,0.1) -- (8,-0.1) node[below] {$8$};
\node[below] at (0,0) {$0$};
\end{tikzpicture}
\caption{Plot of $V^{-1}\mathcal{J}(V)$\label{fig:plot_h_R}}
\end{subfigure}
}
\caption{Plots of $\mathcal{J}(V)$, $\mathcal{J}^2(V)$ and $V^{-1}\mathcal{J}(V)$ for $\mathcal{R}_4$.}
\end{figure}

Finally, we notice that if one were interested in computing the Cheeger constant $h(\mathcal{R}_L)$, one would be led to look for critical points of $V^{-1}\mathcal{J}(V)$, whose plot is shown in \cref{fig:plot_h_R}, for $\mathcal{R}_4$. By~\cite[Theorem~1.4]{LNS17} the inner Cheeger set of $\mathcal{R}_L$ cannot be empty, thus the minimum of $V^{-1}\mathcal{J}(V)$ occurs in the interval $(2L + \pi -4,2L)$. The Cheeger constant of $h(\mathcal{R}_L)$ is by now well-known to be
\[
h(\mathcal{R}_L)= \frac 12 \cdot \frac{4-\pi}{\frac L2 + 1 - \sqrt{(\frac L2 - 1)^2 + \pi \frac L2}}\,,
\]
see for instance~\cite[Section~3.1]{PS17} or the discussion following~\cite[Theorem~3]{KL06} together with the correction done in~\cite[Open problem~1]{Kaw16}. Rationalizing the above ratio, one can easily see that $h(\mathcal{R}_L)$ converges to $1$ as $L\to +\infty$, and precise asymptotic estimates are given in~\cite[Theorem~3.2]{LP16}. Therefore, the length of the interval $[|E_{h_\Om}|, |\mathcal{R}_L|]$ converges to $4-\pi$, while the length of the interval $[\pi, |E_{h_\Om}|]$ diverges. This example shows that with \cref{thm:main_shape_min} we can cover a range of volumes, $[\pi, |\mathcal{R}_L|]$, that can be significantly larger than the range covered by~\cite[Theorem~2.3]{LS20}, $[|E_{h_\Om}|, |\mathcal{R}_L|]$. This can also happen for non convex sets, as the next example shows.

\subsection{Cross-shaped sets}

We consider now a cross-shaped domain $\mathcal{X}_L$, given by the union of two rectangles $\mathcal{R}_L$ in such a way that they share the barycenter and their boundaries meet orthogonally, refer also to \cref{fig:cross-shaped}. 
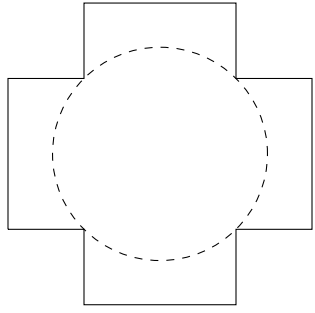
\begin{figure}
\begin{tikzpicture}
\draw (0,0) -- (1,0) -- (1,-1) -- (3, -1) -- (3,0) -- (4,0) -- (4,2) -- (3,2) -- (3, 3) -- (1,3) -- (1,2) -- (0,2) -- (0,0);
\draw[dashed] (2,1) circle (1.414cm);
\end{tikzpicture}
\caption{The cross-shaped set $\chi_L$, for $L=4$, and its inball.}
\label{fig:cross-shaped}
\end{figure}
Assuming that $L\ge 4$, their intersection is a square of side $2$. Hence, it is immediate to see that $\inr(\mathcal{X}_L)=\sqrt{2}$. Then, we are in the following situation:
\begin{itemize}
\item[i)] for volumes $V\le 2\pi$, the minimizer is a ball of radius $\sqrt{\pi^{-1}V}$;
\item[ii)] for volumes $V\in (2\pi, 2\pi+  4]$, any minimizer is given by the intersection of the two rectangles (a square of side $2$) topped by four circular segments of radius $r$ decreasing from $\sqrt{2}$ to $1$ as $V$ increases;
\item[iii)] for volumes $V \in (2\pi +4, 4L+2\pi -12]$, a minimizer is the suitable union of balls of radius $1$, and it is not difficult to show that there exists a unique one which is central-symmetric;
\item[iv)] for volumes $V \in (4L+2\pi -12, 4L -4]$, there exists $\kappa > 1$ such that $E_\kappa = (\mathcal{X}_L)^{\kappa^{-1}} \oplus B_{\kappa^{-1}}$ is the unique minimizer.
\end{itemize}

The corresponding profile is given by
\[
\mathcal{J}(V) =
\begin{cases}
2\sqrt{\pi V}\,, & V\in [0, 2\pi]\\
8r(V)\arcsin\left(\frac{1}{r(V)} \right)\,, &V\in (2\pi, 2\pi+4]\\
2\pi+V-4\,, &V\in (2\pi+4, 4L+2\pi -12]\\
-2\sqrt{2}\sqrt{4-\pi}\sqrt{4L-4-V}+4L\,, &V\in (4L+2\pi -12, 4L-4]
\end{cases}
\]
where $r(V)$ is the unique solution\footnote{Uniqueness follows from the fact that $V(r)$ is strictly increasing, thus the inverse function exists.} in $[1,\sqrt{2}]$ of the equation
\[
V=4\left(1+r^2\arcsin\left(\frac{1}{r}\right) - \sqrt{r^2-1} \right)\,.
\]
In \cref{fig:plots_X} the plots of $\mathcal{J}(V)$, $\mathcal{J}^2(V)$ and of $V^{-1}\mathcal{J}(V)$ are shown for the choice $L=4$. Regarding the Cheeger constant, we recall that the inner Cheeger formula tells us that $r=h_\Om^{-1}$ satisfies the equality
\[
|\Om^r| = \pi r^2\,,
\]
since for $r\ge 1$ we have $|\Om^r| \ge 4-\pi$, and it is immediate to see that the minimum of $V^{-1}\mathcal{J}(V)$ always lies in the fourth interval of definition of $\mathcal{J}$. Thus, just as for the rectangles, one can see that as $L\to +\infty$ the length of the interval $[|E_{h_\Om}|, |\mathcal{X}_L|]$ is bounded from above by $8-2\pi$ (actually, it converges to), while the length of the interval $[2\pi, |E_{h_\Om} |]$ diverges, showing again how relevant the extension provided by \cref{thm:main_shape_min} can be, with respect to previously known results.
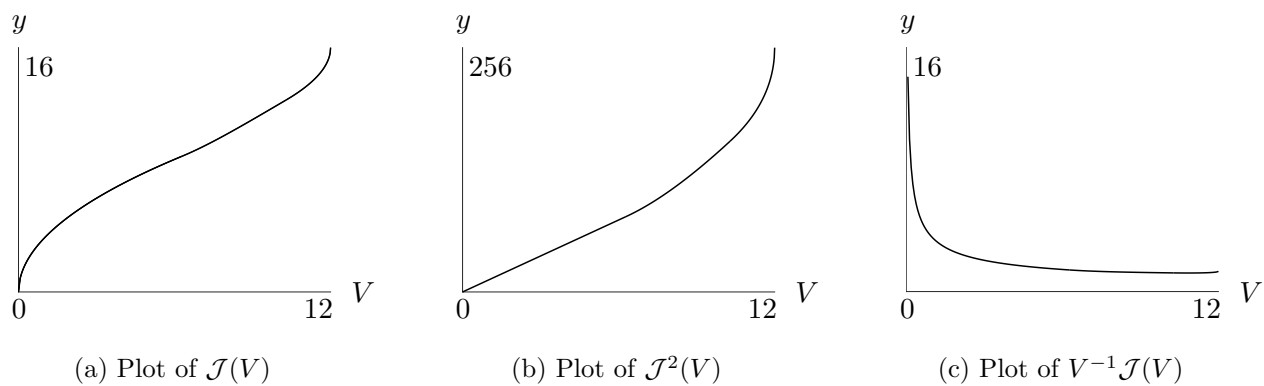
\begin{figure}[t]%
\makebox[\linewidth][c]{
\begin{subfigure}[b]{.33\linewidth}
\vspace{15pt}
\begin{overpic}[width=1\textwidth]
{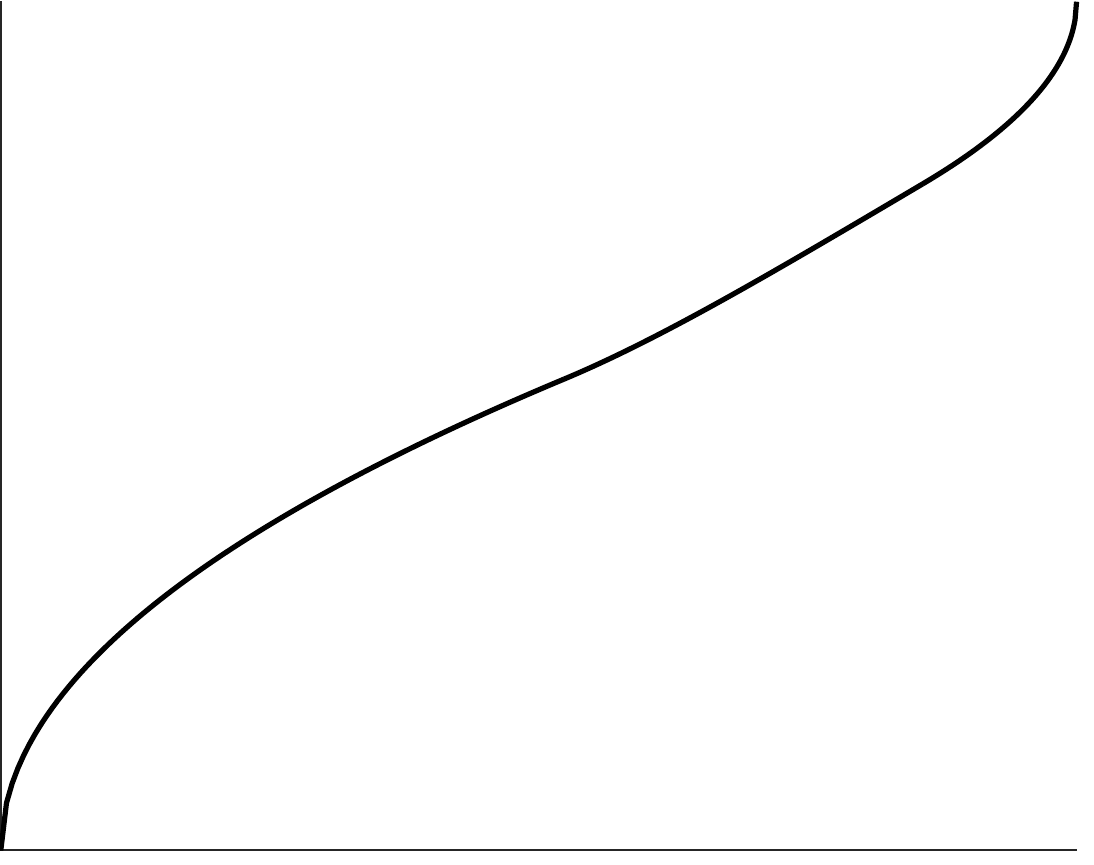}
\put (-2, 85) {$y$}
\put (2, 70) {$16$}
\put (-2, -7) {$0$}
\put (90, -7) {$12$}
\put (105, -2) {$V$}
\end{overpic}
\vspace{0pt}
\caption{Plot of $\mathcal{J}(V)$ \label{fig:plot_J_X}}
\end{subfigure}
\qquad
\qquad
\begin{subfigure}[b]{.33\linewidth}
\begin{overpic}[width=1\textwidth]
{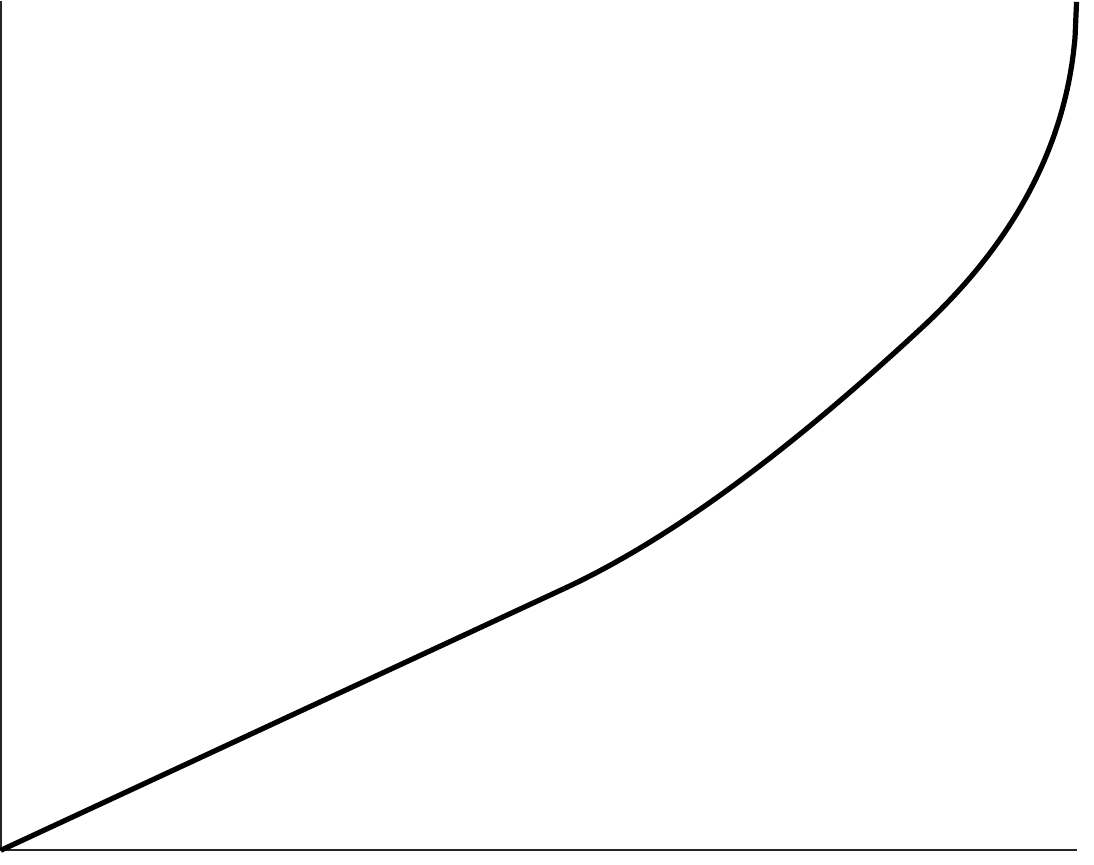}
\put (-2, 85) {$y$}
\put (2, 70) {$256$}
\put (-2, -7) {$0$}
\put (90, -7) {$12$}
\put (105, -2) {$V$}
\end{overpic}
\vspace{0pt}
\caption{Plot of $\mathcal{J}^2(V)$ \label{fig:plot_J^2_X}}
\end{subfigure}
\qquad
\qquad
\begin{subfigure}[b]{.33\linewidth}
\begin{overpic}[width=1\textwidth]
{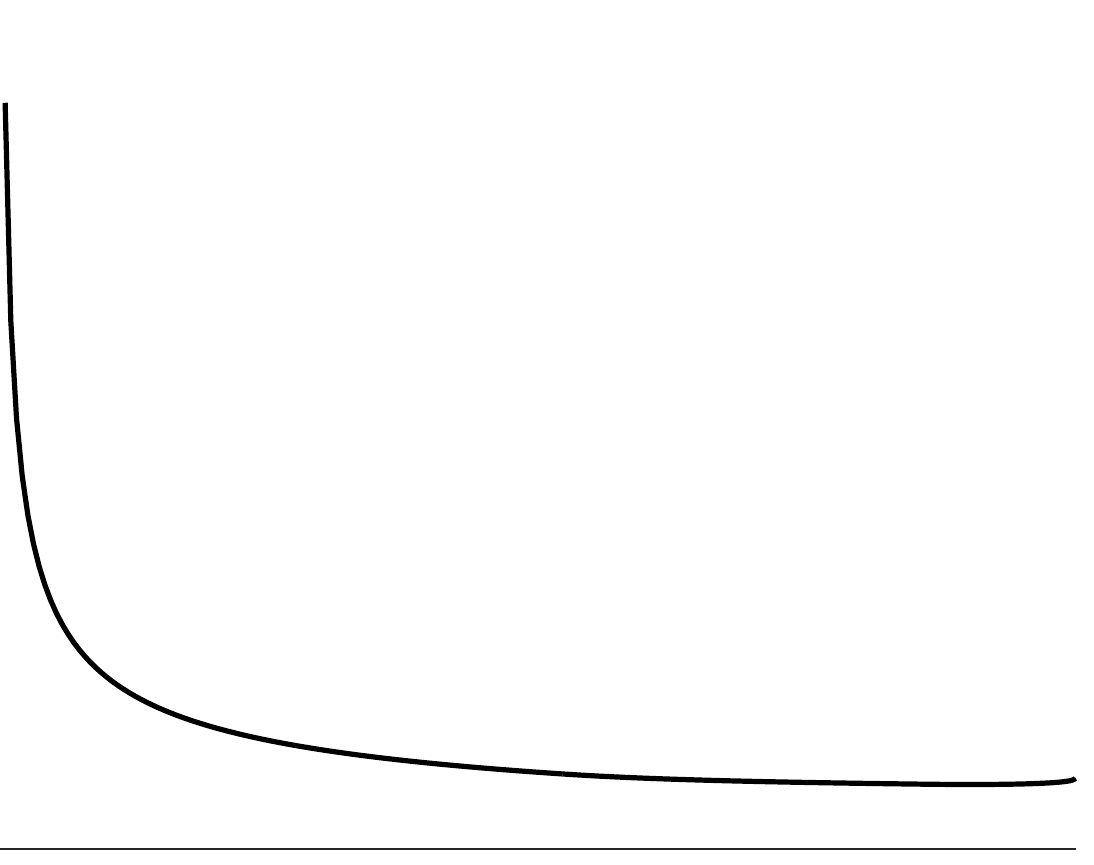}
\put (-2, 85) {$y$}
\put (2, 70) {$16$}
\put (-2, -7) {$0$}
\put (90, -7) {$12$}
\put (105, -2) {$V$}
\end{overpic}
\vspace{0pt}
\caption{Plot of $V^{-1}\mathcal{J}(V)$ \label{fig:plot_h_X}}
\end{subfigure}
}
\caption{Plots of $\mathcal{J}(V)$, $\mathcal{J}^2(V)$ and $V^{-1}\mathcal{J}(V)$ for $\mathcal{X}_4$.\label{fig:plots_X}}
\end{figure}
%



\bibliographystyle{plainurl}

\bibliography{convexity_ip_no_necks}

\end{document}